\documentclass[sn-mathphys,Numbered]{sn-jnl}

\usepackage{graphicx}%
\usepackage{multirow}%
\usepackage{amsmath,amssymb,amsfonts}%
\usepackage{amsthm}%
\usepackage{mathrsfs}%
\usepackage[title]{appendix}%
\usepackage{xcolor}%
\usepackage{textcomp}%
\usepackage{manyfoot}%
\usepackage{booktabs}%
\usepackage{algorithm}%
\usepackage{algorithmicx}%
\usepackage{algpseudocode}%
\usepackage{listings}%
\usepackage{subcaption}
\usepackage{geometry}
\usepackage{amsmath}
\usepackage{bm}

\newtheorem{theorem}{Theorem}
\newcommand{\D}{\mathrm{d}}
\newcommand{\ee}{\mathrm{e}}
\newcommand{\ii}{\mathrm{i}}
\newtheorem{lemma}{Lemma}
\newtheorem{assumption}{Assumption}
\newtheorem*{notation*}{Notations}

\begin{document}

\title[Article Title]{Numerical integrator for highly oscillatory differential equations based on the Neumann series }


\author*[2]{\fnm{Rafa{\l}} \sur{Perczy\'{n}ski}}\email{rafal.perczynski@phdstud.ug.edu.pl}

\author[1]{\fnm{Grzegorz} \sur{Madejski}}

\affil*[1]{Institute of Informatics, Faculty of Mathematics, Physics and Informatics, University of Gda\'{n}sk,  80-308 Gda\'{n}sk, Poland}
\affil[2]{Institute of Mathematics, Faculty of Mathematics, Physics and Informatics, University of Gda\'{n}sk, 80-308 Gda\'{n}sk, Poland}


\abstract{We propose a third-order numerical integrator based on the Neumann series and the Filon quadrature,
designed mainly for highly oscillatory partial differential equations. The method can be applied to
equations that exhibit small or moderate oscillations; however, counter-intuitively, large oscillations
increase the accuracy of the scheme. With the proposed approach, the convergence order of the method can be easily improved. Error analysis of the method is also performed. We consider linear evolution equations involving first- and
second-time derivatives that feature elliptic differential operators, such as the heat equation or the
wave equation. Numerical experiments consider the case in which the space dimension is greater than one and confirm the theoretical study.}




\maketitle

\section{Introduction}\label{sec1}

We consider the following type of highly oscillatory partial differential equations
\begin{align}
\label{eq:1.5}
&\partial_tu(x,t)=\mathcal{L}u(x,t)+f(x,t)u(x,t),\qquad t\in [0,t^\star],\quad x\in\Omega\subset\mathbb{R}^m,\\ \nonumber
&u(x,0)=u_0(x),
\end{align}
with zero boundary conditions, where $\Omega$ is an open and bounded subset of $\mathbb{R}^m$  with smooth boundary $\partial \Omega$, $t^\star>0$ and $\mathcal{L}$ is a linear differential operator of degree $2p$, $p\in \mathbb{N}$, defined on  $\Omega$
\begin{align} \label{diff_operator}
\mathcal{L} = \sum_{|\bm{p}|\leq 2p}^{}a_{\bm{p}}(x) D^{\bm{p}},  \qquad D^{\bm{p}}= \frac{\partial^{p_1}}{\partial x_1^{p_1}}\frac{\partial^{p_2}}{\partial x_2^{p_2}}\dots \frac{\partial^{p_m}}{\partial x_m^{p_m}}, \quad  x\in \Omega\subset \mathbb{R}^m.
\end{align}
Multi-index $\bm{p}$ is an $m$-tuple of nonnegative integers $\bm{p} = (p_1,p_2,\dots,p_m)$ and $a_{\bm{p}}(x)$ are smooth, complex-valued functions of $x \in \bar{\Omega}$.
We assume that function $f(x,t)$ from the equation (\ref{eq:1.5}) is a highly oscillatory of type
\begin{equation}\label{ideal_function_f}
f(x,t) = \sum_{\substack{n=-N \\ n\neq 0}}^{N}\alpha_n(x,t)e^{i n\omega  t}, \quad \omega \gg 1, \quad N\in \mathbb{N},
\end{equation}
where $\alpha_n$ are sufficiently smooth, complex-valued functions.  The aim of this paper is to construct a numerical integrator designed for highly oscillatory equations of type (\ref{eq:1.5}).

Highly oscillatory differential equations are difficult to solve numerically and are of great importance in computational mathematics, which is why they have gained special attention in the field \cite{HLW_2006, hochbruck_ostermann_2010}, and references therein. In particular,
 differential equations of type (\ref{eq:1.5}) with extrinsic high oscillations arise in various fields, including electronic engineering \cite{Condon_2009}, when computing scattering frequencies \cite{Condon_2019}, and in quantum mechanics \cite{IKS_2018,CKLP_2021}. In addition to the previously mentioned works, computational methods dedicated to equations of type (\ref{eq:1.5}) are presented for example in \cite{Longbin_2023,BaderKG,Ait2023,Cai_2022, KL_2023}. In this paper, we present a complementary approach as discussed in the aforementioned papers. Given the generality of equations (\ref{eq:1.5}), the numerical scheme derived in this manuscript can be effectively applied to a range of linear partial differential equations, including the heat equation and the wave equation.


In traditional numerical methods applied to highly oscillating problems, it is typically a requirement that the time step $h$ satisfies $h\omega < 1$. This causes the method to become extremely expensive when $\omega$ is large. This occurs because conventional schemes are constructed using Taylor expansions, where error formulas involve expressions with high derivatives of a highly oscillatory function. On the other hand, methods for highly oscillatory equations based on the Modulated Fourier expansion or the asymptotic expansion may not converge to a solution, rendering them effective only for equations with large oscillatory parameter $\omega$. In this paper, we propose a third-order method whose accuracy improves with increasing parameter $\omega$ and decreasing time step $h$.
  Furthermore, the approach presented in this paper allows for easy improvement of the convergence order of the proposed numerical integrator. 

To provide a more detailed understanding of the challenges associated with the numerical approximation of highly oscillating differential equations, let us apply Duhamel's formula to equation (\ref{eq:1.5}) and write it in the following integral form
\begin{eqnarray}\label{Duhamel_intro}
    u(t+h)&=& \ee^{h\mathcal{L}}u(t) +\int_0^h\ee^{(h-\tau)\mathcal{L}}f(t+\tau)u(t+\tau)\D\tau,
\end{eqnarray}
where  $u_0$ and $f(s), \ u(s)$ for  fixed $s$ are elements of the appropriate Banach spaces.
Let us note that the second time derivative of the solution to equation (\ref{eq:1.5}) satisfies 
$$u''(t) = \mathcal{L}u'(t)+f'(t)u(t)+f(t)u'(t),$$
 therefore $u''(t) = \mathcal{O}(\omega)$ and furthermore $u^{(k)}(t) = \mathcal{O}(\omega^{k-1})$. This implies that approximating the integral from equation (\ref{Duhamel_intro}) using standard quadrature rules, as in basic numerical schemes, leads to a significant error. 
 For that reason, we use a different approach to this issue. In paper \cite{perczyński2023asymptotic}, it was shown that the solution to equation (\ref{eq:1.5}) can be presented as the Neumann series. Subsequently, by expanding asymptotically each integral within the Neumann series, it was demonstrated that the solution of the equation can be expressed as the Modulated Fourier expansion. The Modulated Fourier expansion, also known as an asymptotic expansion or frequency expansion, is a technique used for analyzing highly oscillatory problems. It is comprehensively and thoroughly described in \cite{HLW_2006}.
 By representing the solution as the Neumann series, the time derivatives of the solution, which can be large for highly oscillating equations, do not appear in the error formula of the numerical scheme. In this paper, instead of employing an asymptotic expansion for the integrals within the Neumann series (which is effective only in cases of high oscillations), we approximate them using quadrature rules designed for highly oscillatory integrals, such as Filon-type methods. By this approach, we can provide that the local error of the presented numerical scheme can be estimated by $Ch^4$, where constant $C$ is independent of time step $h$ and parameter $\omega$.
Furthermore, when considering a potential function $f$ with only positive frequencies, i.e. when only numbers $n>0$ appear in formula (\ref{ideal_function_f}), one can show that the local error is bounded by $C\min\left\{h^4, h^2\omega^{-2}, \omega^{-3}\right\}$, where again the constant $C$ is independent of both $h$ and $\omega$. 

The convergence rate of the method can be easily improved by approximating a greater number of integrals from the Neumann series. However, this enhancement comes at the cost of requiring better regularity for both the initial condition $u_0$ and the functions $\alpha_n$, and also leads to increased computational complexity.

The paper is organized as follows. In Section \ref{Preliminaries}, we introduce  the two fundamental ingredients of the proposed method -- the Neumann series and the Filon
	method. Section \ref{sec2} provides the derivation of the proposed numerical integrator.   Section \ref{sec3} is dedicated to the error analysis of the method. In Section \ref{sec4}, we demonstrate the application of the scheme to equations involving a second time derivative, while Section \ref{sec5} presents the results of numerical experiments.

\section{Preliminaries}\label{Preliminaries}
In this section, we briefly introduce the basic tools needed to build the proposed numerical method: the Neumann series and the Filon method.
\subsection*{The Neumann series}
Let us consider the following ordinary differential equation
\begin{eqnarray}\label{ODE}
Y'(t) = A(t)Y(t), \qquad Y(0)= Y_0,
\end{eqnarray}
where $Y: \mathbb{R} \to \mathbb{C}^n$ and $A(t)$ is an $n \times n$ time-dependent matrix. 
Equation (\ref{ODE}) can be written in the following form
\begin{eqnarray}\label{int_ode}
Y(t) = Y_0+\int_{0}^{t}A(\tau)Y(\tau)\D \tau.
\end{eqnarray}
By iterating  equation (\ref{int_ode}), one can show that the solution of the problem (\ref{int_ode}) is given by the  series
\begin{eqnarray}\label{ODE_series}
Y(t) = \sum_{d=0}^{\infty}T^dY_0,
\end{eqnarray}
where $$T^d Y_0= \int_0^t A(\tau_1)\int_{0}^{\tau_1}A(\tau_2)\dots \int_{0}^{\tau_{d-1}}A(\tau_d)Y_0\D\tau_d\dots\D\tau_1.$$
The series (\ref{ODE_series}) is known as the Neumann series and the Dyson series \cite{iserles_nemann}, and it converges to the solution of equation (\ref{int_ode}) for all values of $t$ provided that the matrix $A(t)$ is bounded \cite{Blanes_casas}.
\subsection*{The Filon method}
The Filon method is a quadrature rule designed for highly oscillatory integrals.
Suppose we wish to approximate the following integral
\begin{eqnarray}
I[h,(a,b)] = \int_a^b h(s)\ee^{\ii\omega g(s)}\D s,
\end{eqnarray}
where $h$ and $g$ are real-valued, sufficiently smooth functions, $h\neq 0$ in $[a,b]$ and $\omega \gg 1$.  Numerical approximation of such integrals by standard methods based on the Taylor expansion leads to a significant error.
Consider the Hermite interpolation polynomial $p$ that approximates the function $h$, $p(s)\approx h(s)$. Let $p$ satisfy the following conditions
$$p^{(k)}(a) = h^{(k)}(a), \quad p^{(k)}(b) = h^{(k)}(b), \quad k=0,1,\dots N.$$
We assume that the moments
$$\mu_k=\int_a^b s^k\ee^{\ii\omega g(s)} \D s ,\quad k= 0,1,\dots 2N+1$$
can be calculated explicitly.
The Filon method reads
\begin{equation*}
\int_a^b h(s)\ee^{\ii\omega g(s)} \D s \approx \int_a^b p(s) \ee^{\ii\omega g(s)}\D s.
\end{equation*}
The Filon method is very effective in the approximation of highly oscillatory integrals. Additionally, for small values of $\omega$, it behaves similarly to standard quadrature rules. \\
Filon method may be used in approximation of multivariate highly oscillatory integrals. 
Various modifications of the Filon quadrature are possible; see, for example, \cite{iserles_book}.

\section{Derivation of the method} \label{sec2}
For the convenience of presenting the method, we introduce the necessary notation and make the following general assumption,  which will be used throughout the paper.
 \begin{notation*}
By  $ H^{2p}(\Omega) =W^{2p,2}(\Omega) $, where $p$ is a nonnegative integer, we understand the Sobolev space equipped with standard norm $\| \ \ \|_{H^{2p}(\Omega)}$, and $H^{p}_0(\Omega)$ is the closure of $C_0^{\infty}(\Omega)$ in the space $H^p(\Omega)$. By $u[t](\tau)$ we understand function $u$ such that $u[t](\tau) = u(t+\tau)$.
We slightly abuse the notation and also denote $u(t):=u( \cdot  ,t)$ as an element of an appropriate Banach space. Throughout the text, by  $\| \ \ \|:=\| \ \ \|_{L^2(\Omega)}$ we understand the standard norm of $L^2(\Omega)$ space.
\end{notation*}

\begin{assumption} \label{assumption1}
Suppose that
\begin{enumerate}
    \item $\Omega $ is an open and bounded set in $\mathbb{R}^m$ with smooth boundary $\partial \Omega$.
    \item Operator  $-\mathcal{L}: D(\mathcal{L}):= H_0^p(\Omega)\cap H^{2p}(\Omega) \rightarrow L^2(\Omega)$, where $\mathcal{L}$  is of form (\ref{diff_operator}), is a strongly elliptic of order $2p$   and has smooth, complex-valued coefficients $a_{\textbf{p}}(x)$.  Moreover $2p>m/2$.
    \item $u_0 \in D(\mathcal{L}^4)$ and $\alpha_n \in C^4\left([0,t^\star ] ,H^{8p}(\Omega)\right)$, \ $n = -N,\dots,-1,1,\dots,N,$ where  \\
    $ D(\mathcal{L}^k)=\{u\in D(\mathcal{L}^{k-1}): \mathcal{L}^{k-1}u \in D(\mathcal{L})\}$, $k=2,\dots$
\end{enumerate}
\end{assumption} 
The assumed regularity of functions $u_0$ and $\alpha_n$ is related to the accuracy of the method.
 Assumption \ref{assumption1} ensures that differential operator $\mathcal{L}$  is the infinitesimal generator of  a strongly continuous semigroup $\{\ee^{t\mathcal{L}}\}_ {t\geq 0}$ on $L^2(\Omega)$ and therefore $\max_{t \in [0,t^\star]}\|\ee^{t\mathcal{L}}\|_{L^2(\Omega)\leftarrow L^2(\Omega)}\leq C(t^\star)$, where $C(t^\star)$ is some constant independent of $t$ \cite{PAZY}. 
 
We wish to build the method based on time steps. Therefore, based on the semigroup property, we can modify equation (\ref{eq:1.5}) and express it as follows
\begin{align}
\label{eq:1.5_steps}
&\partial_su[t](x,s)=\mathcal{L}u[t](x,s)+f[t](x,s)u[t](x,s),\qquad s\in [0,h],\quad x\in\Omega\subset\mathbb{R}^m,\\ \nonumber
&u[t](x,0)=u(x,t),
\end{align}
where $t\geq 0$ and $h>0$ is a  small time step.
By $u[t](x,s)$ we understand $u[t](x,s)= u(x,t+s)$.
By applying Duhamel formula to (\ref{eq:1.5_steps}), we obtain
\begin{eqnarray}\label{eq_duhamel}
    u[t](s)&=& \ee^{s\mathcal{L}}u[t](0) +\int_0^s\ee^{(s-\tau)\mathcal{L}}f[t](\tau)u[t](\tau)\D\tau.
\end{eqnarray}
We could simply write  $u[t](s)= u(t+s)$, but the above notation helps avoid misunderstandings in subsequent formulas.

 Let $V_t$ denotes the following space
\begin{eqnarray*}\label{V_space}
    V_t := C\left([t,t+h],L^2(\Omega)\right), \quad t\geq 0,\quad h >0,
\end{eqnarray*} 
Define the linear operator  $T_t:V_t\rightarrow V_t$ 
$$T_tu[t](s) = \int_0^s\ee^{(s-\tau)\mathcal{L}}f[t](\tau) u[t](\tau)\D\tau, \quad s\in[0,h],$$
where function $f$ is defined in (\ref{ideal_function_f}). 
The Neumann series for equation (\ref{eq_duhamel}) reads
\begin{eqnarray}\label{Neumann_series}
    u[t](h) = \sum_{d=0}^{\infty}T^d_t\ee^{h\mathcal{L}}u[t](0).
\end{eqnarray}
 Term  $T^d_t\ee^{h\mathcal{L}}u[t](0), \ d=1,2,\dots$,  from (\ref{Neumann_series}) is equal to
\begin{eqnarray*}\label{T^d}
T^d_tv[t](h)=
\int_0^h\ee^{(h-\tau_d)\mathcal{L}}f[t](\tau_d)\int_0^{\tau_d}\ee^{(\tau_d-\tau_{d-1})\mathcal{L}}f[t](\tau_{d-1})\dots\int_0^{\tau_2}\ee^{(\tau_2-\tau_1)\mathcal{L}}f[t](\tau_1)\ee^{\tau_1\mathcal{L}}u[t](0)\D\tau_1\dots\D\tau_d.
\end{eqnarray*}
It can be shown that the  series (\ref{Neumann_series}) converges in the norms $\| \ \ \|_{L^2(\Omega)}$ and $\| \ \ \|_{H^{2p}(\Omega)}$ to the solution of equation (\ref{eq_duhamel}), where $2p >m/2$, for arbitrary time variable $h>0$ \cite{perczyński2023asymptotic}.
The idea for finding an approximate solution to equation (\ref{eq:1.5}) involves approximating the first $r$ terms of the Neumann series (\ref{Neumann_series}) using quadrature methods designated to highly oscillatory integrals. 
For convenience, we introduce a set 
\begin{eqnarray}\label{set_N}
\bm{N}^d:= \{-N,-N+1,\dots,-1,1,\dots,N-1,N\}^d \subset \mathbb{N}^d,
\end{eqnarray}
where  $2N$ is a number of terms in sum (\ref{ideal_function_f}).
Using definition (\ref{ideal_function_f}) of the function $f$ and the linearity of semigroup operator, we can write each term of the Neumann series $T^d_t\ee^{h\mathcal{L}}u[t](0), \ \D=1,2,\dots$ in a more convenient form for our considerations 
\small
\begin{eqnarray*}
 T^d_t\ee^{h\mathcal{L}}u[t](0)&=& \sum_{n_1,\dots,n_d\in \bm{N}^1}\int_{\sigma_d(h)}F_{\bm{n}}(h,\tau_1,\dots,\tau_d)\ee^{\ii\omega(  n_1(\tau_1+t)+\dots +n_d(\tau_d+t)) }\D \tau_1\dots\tau_d = \sum_{\bm{n}\in \bm{N}^d} \ee^{\ii\omega t\bm{n}^T \bm{1} }I[F_{\bm{n}},\sigma_d(h)], 
\end{eqnarray*}
\normalsize
where
\begin{eqnarray} \label{function_f_operator}
 F_{\bm{n}}(h,\tau_1,\dots\tau_d) &=& \ee^{(h-\tau_d)\mathcal{L}}\alpha_{n_d}[t](\tau_d)\ee^{(\tau_d-\tau_{d-1})\mathcal{L}}\alpha_{n_{d-1}}[t](\tau_{d-1})\dots \ee^{(\tau_{2}-\tau_1)\mathcal{L}}\alpha_{n_1}[t](\tau_1)\ee^{\tau_1\mathcal{L}}u[t](0),\\
 I[F_{\bm{n}},\sigma_d(h)]&=&  \int_{\sigma_d(h)}F_{\bm{n}}(h,\bm{\tau})\ee^{\ii\omega \bm{n}^T \bm{\tau} }\D \bm{\tau}, \nonumber \\
 \bm{\tau}&=& (\tau_1,\tau_2,\dots,\tau_d), \quad \bm{1} = (1,1,\dots,1),\nonumber
 \end{eqnarray}
 and $\sigma_d(h)$ denotes a $d$-dimensional simplex
 $$\sigma_d(h)=\{\bm{\tau}:=(\tau_1,\tau_2,\dots,\tau_d)\in \mathbb{R}^d:h\geq \tau_d\geq \tau_{d-1}\geq \dots \geq \tau_2\geq \tau_1 \geq 0\}.$$
Using the above notation,  solution $u$ of (\ref{eq_duhamel}) can be written as
 \begin{eqnarray}\label{Neumann_integrals}
 u(t+h)=u[t](h) = \ee^{h\mathcal{L}}u[t](0)+\sum_{d=1}^{\infty}\ee^{\ii\omega t\bm{n}^T \bm{1}}\sum_{\bm{n}\in \bm{N}^d} I[F_{\bm{n}},\sigma_d(h)].
 \end{eqnarray}
In the proposed numerical scheme, for each time step $h$ we take the first four terms of the above series that approximate the function $u(t+h)$
\begin{eqnarray*}\label{Neumann_integrals_sheme}
 u(t+h)\approx \ee^{h\mathcal{L}}u[t](0)+\sum_{d=1}^{3}\sum_{\bm{n}\in \bm{N}^d} \ee^{\ii\omega t\bm{n}^T \bm{1}}I[F_{\bm{n}},\sigma_d(h)].
 \end{eqnarray*}
Then, we approximate each integral $I[F_{\bm{n}},\sigma_d(h)]$ in the above sum by applying the Filon quadrature. As a result, we derive a fourth-order local method.
By employing Filon quadrature, the method's error converges to zero both as $h\to0$ and as $\omega\to \infty$.
Our decision to consider only the first four terms in the Neumann expansion is rather arbitrary. The method can be enhanced to achieve a higher level of accuracy, albeit with increased computational costs and the requirement of better regularity for functions $u_0$ and $\alpha_n$.

Consider function $F(\tau):=F_{n_1}(h,\tau)$  from the second term of the Neumann series and the following univariate integral 
\begin{eqnarray*}\label{univariate_int}
    \int_0^h F(\tau)\ee^{n_1\ii\omega \tau}\D \tau,
\end{eqnarray*}
where $n_1 =-N,-N+1,\dots,-1,1,\dots,N$.  Let $p(\tau)$ be a  cubic Hermite interpolating polynomial
\begin{eqnarray*}\label{hermite}
F(\tau) \approx p(\tau) = F(0)+a_{1,1}\tau+a_{1,2}\tau^2+a_{1,3}\tau^3, 
\end{eqnarray*}
 which satisfy the conditions:
$p(0)=F(0)$, $p(h)=F(h)$, $p'(0)=F'(0)$, $p'(h)=F'(h)$.
We have
\begin{eqnarray*}\label{uni_filon_app}
     \int_0^h F(\tau)\ee^{\ii n_1 \omega \tau}\D \tau \approx \int_0^h  p(\tau) \ee^{\ii n_1\omega\tau} \D\tau,
\end{eqnarray*}
and the moments 
$$\int_0^h \tau^k \ee^{\ii n_1\omega \tau}\D\tau, \quad k=0,1,2,3,$$
can be calculated explicitly.
Let now $F(\tau_1,\tau_2):=F_{\bm{n}}(h,\tau_1,\tau_2), \ \bm{n}\in \bm{N}^2$, and consider a bivariate integral
\begin{eqnarray} \label{int_bi}
    \int_0^h \int_0^{\tau_2}F(\tau_1,\tau_2)\ee^{\ii\omega(n_1\tau_1+n_2\tau_2)} \D\tau_1\D\tau_2.
\end{eqnarray}
We approximate function $F(\tau_1,\tau_2)$ in points $(0,0)$, $(0,h)$ and $(h,h)$, the vertices of the simplex $\sigma_2(h)$, by linear function $p(\tau_1,\tau_2)$, 
\begin{eqnarray*}\label{bivariate_approx}
F(\tau_1,\tau_2)\approx p(\tau_1,\tau_2)=F(0,0)+a_{2,1}\tau_1+a_{2,2}\tau_2.
\end{eqnarray*}
 The approximation of integral (\ref{int_bi}) by the Filon quadrature rule reads 
\begin{eqnarray*}
    \int_0^h \int_0^{\tau_2}F(\tau_1,\tau_2)\ee^{\ii\omega(n_1\tau_1+n_2\tau_2)} \D\tau_1\D\tau_2\approx\int_0^h \int_0^{\tau_2}p(\tau_1,\tau_2)\ee^{\ii\omega(n_1\tau_1+n_2\tau_2)} \D\tau_1\D\tau_2,
\end{eqnarray*}
and the integral on the right-hand side can be computed explicitly.
Similarly, we proceed with the triple integral. Function $F(\tau_1,\tau_2,\tau_3):= F_{\bm{n}}(h,\tau_1,\tau_2,\tau_3), \ \bm{n} \in \bm{N}^3$ is approximated by linear function $p$ at the vertices of the simplex $\sigma_3(h)$: 
$(0,0,0)$, $(0,0,h)$, $(0,h,h)$, $(h,h,h),$
\begin{eqnarray*}
F(\tau_1,\tau_2,\tau_3)\approx p(\tau_1,\tau_2,\tau_3)= F(0,0,0)+a_{3,1}\tau_1+a_{3,2}\tau_2+a_{3,3}\tau_3.
\end{eqnarray*}
Then we have
\begin{small}
\begin{eqnarray*}\label{Filon_triple_integral}
    \int_0^h \int_0^{\tau_3}\int_0^{\tau_2}F(\tau_1,\tau_2,\tau_3)\ee^{\ii\omega(n_1\tau_1+n_2\tau_2+n_3\tau_3)} \D\tau_1\D\tau_2\D\tau_3\approx\int_0^h \int_0^{\tau_3}\int_0^{\tau_2}p(\tau_1,\tau_2,\tau_3)\ee^{\ii\omega(n_1\tau_1+n_2\tau_2+n_3\tau_3)} \D\tau_1\D\tau_2\D\tau_3.
\end{eqnarray*}
\end{small}
The precise formulas for determining the coefficients $a_{i,j}$ are presented in the Appendix \ref{secA1}.

The proposed algorithm for computing the successive approximation of the solution $u$ can be expressed in the following form:
\begin{eqnarray}\label{scheme}
    u^{k+1} &=& \Bigg( \ee^{h\mathcal{L}}+ \sum_{n_1}  \int_0^h \left(a_{1,0}+a_{1,1}\tau+ a_{1,2}\tau^2+a_{1,3}\tau^3 \right)\ee^{n_1\ii\omega( \tau+t_k)}\D\tau \nonumber \\ &+& \sum_{n_1,n_2}\int_0^h \int_0^{\tau_2}(a_{2,0}+a_{2,1}\tau_1+a_{2,2}\tau_2)\ee^{\ii\omega(n_1(\tau_1+t_k)+n_2(\tau_2+t_k))}\D\tau_1\D\tau_2\\ &+&\sum_{n_1,n_2,n_3} \int_0^h \int_0^{\tau_3}\int_0^{\tau_2}(a_{3,0}+a_{3,1}\tau_1+a_{3,2}\tau_2+a_{3,3}\tau_3)\ee^{\ii\omega(n_1(\tau_1+t_k)+n_2(\tau_2+t_k)+n_3(\tau_3+t_k))}\D\tau_1\D\tau_2\D\tau_3\Bigg)u^k\nonumber\\
    t_{k+1} &=& t_{k}+h, \quad k=0,1,\dots,K-1,\nonumber
\end{eqnarray}
where $u^0= u_0$, $t_0=0$, $t_K=t^\star$, $n_1,n_2,n_3 \in \{-N,-N+1,\dots,-1,1,\dots,N\}$ and the coefficients $a_{i,j}$ are chosen so that the corresponding polynomial satisfies the Hermite interpolation conditions. Each of the integrals appearing in the scheme is computed explicitly. Furthermore, the expression $\ee^{h\mathcal{L}}$ and the coefficients $a_{i,j}$ of the interpolating polynomials, after spatial discretization, can be computed very efficiently and accurately using spectral methods \cite{trefethen} and/or splitting methods \cite{Splitting_acta}.

\section{Local error analysis}\label{sec3}
 The entire error of the method comes from two sources: the approximation of each integral from the partial sum of the Neumann series, and the error associated with the truncation of the Neumann expansion. 
In \cite{perczyński2023asymptotic}, the authors provide the asymptotic expansion of integral $I[F_{\bm{n}},\sigma_d(h)]$ from the Neumann series (\ref{Neumann_integrals}), where $F_{\bm{n}}$ is the function of the form (\ref{function_f_operator}), for the special case when the potential function $f$ has positive frequencies, specifically when  $f$ takes the form 
\begin{eqnarray}\label{positive_frequencies}
    f(x,t) = \sum_{\substack{n=1 }}^{N}\alpha_n(x,t)e^{i n\omega  t}, \quad \omega \gg 1, \quad N\in \mathbb{N}.
\end{eqnarray}
In such a situation,
each integral $I[F_{\bm{n}},\sigma_d(h)]$ satisfies the nonresonance condition and therefore can be approximated by the partial sum $\mathcal{S}^{(d)}_r(h)$ of the asymptotic expansion 
\begin{eqnarray*}\label{series+error}
 	I[F_{\bm{n}},\sigma_d(h)]=	\int_{\sigma_d(h)}F(h,\bm{\tau})\ee^{\ii\omega \bm{n}^T\bm{\tau} }\D \bm{\tau} = \mathcal{S}^{(d)}_r(h)+E^{(d)}_r(h), \quad r\geq d,
 	\end{eqnarray*}
 	where
$E^{(d)}_r(h)=\mathcal{O}(\omega^{-r-1})$ is the error related to approximation of integral $I[F_{\bm{n}},\sigma_d(h)]$ by sum $\mathcal{S}^{(d)}_r(h)\sim \mathcal{O}(\omega^{-d})$.
A similar result was first obtained in \cite{IN_2006}, where the authors provided the asymptotic expansion of a multivariate highly oscillatory integral over a regular simplex. However, in our analysis, the non-oscillatory function $F_{\bm{n}}$ is vector-valued rather than real-valued.
We begin the error analysis of the proposed numerical method by considering function $f$ from equation (\ref{eq:1.5}) in the form (\ref{positive_frequencies}).
Recall that $\| \ \ \|$ denotes the standard  norm of $L^2(\Omega)$ space.  In the following estimations, $C$ is some  constant that depends on functions $\alpha_n$, initial condition $u[t](0)$ of equation (\ref{eq:1.5_steps}) , their derivatives, solution $u$, differential operator $\mathcal{L}$ and $t^\star$, but it is independent of the time step $h$ and the oscillatory parameter $\omega$.
Let us also note that since by Assumption \ref{assumption1}, the function $f \in C^4\left([0,t^\star ] ,H^{8p}(\Omega)\right)$, we can apply the Sobolev embedding theorem to conclude that $\|f(s)\|_\infty < \infty$ for all $s \in [0,t^\star]$. Therefore, the norm of product of two functions $f$ and $u$ can easily be estimated as $\|f(s)u(s)\| \leq  \|f(s)\|_\infty \|u(s)\|, \ \forall s\in [0,t^\star]$.

 \subsection{Positive frequencies}
\begin{lemma}\label{lemma1}
    Let $F(\tau)$ be a 4 times continuously differentiable, vector-valued function, and let $p(\tau)$ be a cubic Hermite interpolation polynomial such that $p(0)=F(0)$, $p(h)=F(h)$, $p'(0)=F'(0)$, $p'(h)=F'(h)$.
    Then the error of the Filon method satisfies 
   \begin{eqnarray*}
     \left\|\int_0^h (F-p)(\tau)\ee^{\ii n_1 \omega \tau}\D \tau \right\|  \leq  C\min\left\{h^5, \frac{1}{\omega^3}, \frac{ h^3}{\omega^2}\right\}.
\end{eqnarray*}
\end{lemma}
\begin{proof}
   The estimation that the error is bounded by $C\omega^{-3}$ directly follows from well-known results concerning Filon quadrature, as described in \cite{IN_2006}.
     By using the Taylor series with the remainder in integral form, one can show that $\|F(\tau)-p(\tau)\|\leq C h^4$ and $\|F''(\tau)-p''(\tau)\|\leq C h^2$. Therefore, by using integration by parts, we have
    \begin{eqnarray*}
     \left\|\int_0^h (F(\tau)-p(\tau))\ee^{\ii n_1 \omega \tau}\D \tau \right\| = \frac{1}{(n_1 \omega)^2}\left\|\int_0^h (F''(\tau)-p''(\tau))\ee^{\ii n_1\omega \tau}\D \tau \right\|  \leq C\frac{h^3}{\omega^2},
\end{eqnarray*}
which completes the proof.
\end{proof}
\begin{lemma}\label{lemma2}
     Let $F(\tau_1,\tau_2)$ be a vector-valued function of class $C^2$  and $p(\tau_1,\tau_2)$ be a linear function that satisfies the  conditions: $p(0,0)=F(0,0)$, $p(0,h)=F(0,h)$, $p(h,h)=F(h,h)$. Let numbers $n_1>0$, $n_2>0$. Then 
     \begin{eqnarray*}
    \left\|\int_0^h\int_0^{\tau_2}(F-p)(\tau_1,\tau_2)\ee^{\ii\omega(n_1 \tau_1+n_2\tau_2)}\D\tau_1\D\tau_2\right\| \leq   C\min\left\{h^4, \frac{h^2}{\omega^2}, \frac{1}{\omega^3}\right\}.
     \end{eqnarray*}
\end{lemma}
 \begin{proof}
Since vector $(n_1,n_2)$ satisfies the nonresonance condition, the approximated integral $I[F,\sigma_2(h)]$ can be expanded asymptotically $I[F,\sigma_2(h)]\sim \mathcal{O}(\omega^{-2})$, and therefore the Filon method provides that the error satisfy $I[(F-p),\sigma_2(h)]=\mathcal{O}(\omega^{-3})$ \cite{IN_2006}. As in the case in the proof of Lemma \ref{lemma1}, by using the Taylor series with the remainder in  integral form, we have the estimations $\|F(\tau_1,\tau_2)-p(\tau_1,\tau_2)\|\leq C h^2$ and $\|\partial_{\tau_1}^1\left(F(\tau_1,\tau_2)-p(\tau_1,\tau_2)\right)\|\leq C h$. 
For simplicity, let us assume that $n_1=n_2=1$. Using  integration by parts, we get
 \begin{eqnarray*}
    \left\|\int_0^h\int_0^{\tau_2}(F-p)(\tau_1,\tau_2)\ee^{\ii\omega( \tau_1+\tau_2)}\D\tau_1\D\tau_2\right\| &\leq&  \frac{1}{\omega}\left\|\int_0^h  (F-p)(\tau_2,\tau_2)\ee^{2\ii\omega\tau_2}-(F-p)(0,\tau_2)\ee^{\ii\omega\tau_2}\D\tau_2\right\|\\
    &+& \frac{1}{\omega^2}\left\|\int_0^h  \partial_{\tau_1}^1(F-p)(\tau_2,\tau_2)\ee^{2\ii\omega\tau_2} -\partial_{\tau_1}^1(F-p)(0,\tau_2)\ee^{\ii\omega\tau_2}\D\tau_2\right\|\\
    &+&\frac{1}{\omega^2}\left\|\int_0^h\int_0^{\tau_2}\partial_{\tau_1}^2(F-p)(\tau_1,\tau_2)\ee^{\ii\omega( \tau_1+\tau_2)}\D\tau_1\D\tau_2\right\|.
     \end{eqnarray*}
     The second and third term on the right side of the above inequality are bounded by $Ch^2\omega^{-2}$, where $C$ is some constant independent of $h$ and $\omega$.
     In the case of the first expression, we again apply integration by parts and the definition of the polynomial $p$, and thus get the following
     \begin{eqnarray*}
         &&\frac{1}{\omega}\left\|\int_0^h  (F-p)(\tau_2,\tau_2)\ee^{2\ii\omega\tau_2}-(F-p)(0,\tau_2)\ee^{\ii\omega\tau_2}\D\tau_2\right\|\leq\\
         &&\frac{1}{2\omega^2}\left\|\int_0^h\partial_{\tau_2}^1(F-p)(\tau_2,\tau_2)\ee^{2\ii\omega\tau_2}\D\tau_2\right\|
         +\frac{1}{\omega^2}\left\|\int_0^h\partial_{\tau_2}^1(F-p)(0,\tau_2)\ee^{\ii\omega\tau_2}\D\tau_2\right\|
         \leq C\frac{h^2}{\omega^2},
     \end{eqnarray*}
     which concludes the proof.

\end{proof}
\noindent In a similar vein, we estimate the error of the Filon method for the triple integral
\begin{lemma}\label{lemma3}
    Let $F(\tau_1,\tau_2,\tau_3)$ be a vector valued function of class $C^2$ and $p(\tau_1,\tau_2,\tau_3)$  be a linear function approximating $F$ such that $p(0,0,0)=F(0,0,0)$, $p(0,0,h)=F(0,0,h)$, $p(0,h,h)=F(0,h,h)$, $p(h,h,h)=F(h,h,h)$. Let numbers $n_1,n_2,n_3>0$.
    Then the error of the Filon method can be estimated as follows
    \begin{eqnarray*}
    \left\|\int_0^h \int_0^{\tau_3}\int_0^{\tau_2}(F-p)(\tau_1,\tau_2,\tau_3)\ee^{\ii\omega(n_1\tau_1+n_2\tau_2+n_3\tau_3)} \D\tau_1\D\tau_2\D\tau_3\right\|\leq C\min\left\{h^5,\frac{h^3}{\omega^2},\frac{1}{\omega^4}\right\}.
\end{eqnarray*}
\end{lemma}

To complete the analysis of the local error, we  need to estimate the truncation error of the Neumann series. We write the solution of (\ref{eq_duhamel}) as
\begin{eqnarray*}
    u[t](h) &=& \underbrace{\sum_{d=0}^{r}T^d_t\ee^{h\mathcal{L}}u[t](0)}_{=:u^{[r]}[t](h)}+ \underbrace{\sum_{d=r+1}^{\infty}T^d_t\ee^{h\mathcal{L}}u[t](0)}_{=:\mathcal{R}^{[r+1]}[t](h)} = u^{[r]}[t](h) + \mathcal{R}^{[r+1]}[t](h),
\end{eqnarray*}
where, in our considerations, we take  $r=3$.
\begin{lemma}\label{lemma4}
Let the function $f$ from equation (\ref{eq_duhamel}) be of the form (\ref{positive_frequencies}).
Then
    the remainder $\mathcal{R}^{[4]}[t](h)$ of the Neumann series (\ref{Neumann_series}) satisfied the following estimate
    \begin{eqnarray*}
        \|\mathcal{R}^{[4]}[t](h)\| \leq C\min\left\{h^{4},\frac{h^{2}}{\omega^{2}},\frac{1}{\omega^4}\right\},
    \end{eqnarray*}
    where constant $C$ depends on functions  $\alpha_n$, $u[t](0)$ their derivatives, solution $u$, operator $\mathcal{L}$ and $t^\star$, but is independent of time step $h$ and parameter $\omega$.
\end{lemma}
\begin{proof}
By using the basic properties of the operator norm and the fact that $\|f\|_{\infty}<\infty$ we have 
    \begin{eqnarray*}
 \left\|\mathcal{R}^{[4]}[t](h)\right\| &=& \left\|\sum_{d=4}^{\infty}T^d_t\ee^{h\mathcal{L}}u[t](0) \right\|=\left\|T_t^{4}\sum_{d=0}^{\infty}T^d_t\ee^{h\mathcal{L}}u[t](0)\right\| = \|T_t^4u[t](h)\|\leq C h^{4}\sup_{s\in[0,h]}\|u[t](s)\|. 
    \end{eqnarray*}
Let us now denote by $\bm{T}_t$ the operator $\bm{T}_t = \sum_{d=0}^{\infty}T_t^{d}$.
On the other hand, we estimate
\begin{eqnarray*}
\left\|\mathcal{R}^{[4]}[t](h)\right\| &=& \left\|\sum_{d=4}^{\infty}T^d_t\ee^{h\mathcal{L}}u[t](0) \right\| = \left\|\sum_{d=0}^{\infty}T^d_tT^{4}_t\ee^{h\mathcal{L}}u[t] (0)\right\| =  \left\|\mathbf{T}_tT^{4}_t\ee^{h\mathcal{L}}u[t] (0)\right\| \leq \\
&\leq&\sup_{\|v[t](h)\|\leq 1}\|\mathbf{T}_t v[t](h)\| \|T_t^{4}\ee^{h\mathcal{L}}u[t](0)\|.
\end{eqnarray*}
Expression $T_t^{4}\ee^{h\mathcal{L}}u[t](0)$ is a sum of highly oscillatory integrals over a $4$-dimensional simplex which satisfy the nonresonance condition and therefore  $\|T_t^{4}\ee^{h\mathcal{L}}u[t](0)\| = \mathcal{O}(\omega^{-4})$.
 In addition, by using basic properties of the operator norm and the simple inequality $\|uv\|_{L^2} \leq \|u\|_{L^2}\|v\|_{\infty}$, we have
 \begin{eqnarray*}
     \|T^4_t\ee^{h\mathcal{L}}u[t](0)\|&=&
\left\|\int_0^h\ee^{(h-\tau_4)\mathcal{L}}f[t](\tau_4)\int_0^{\tau_4}\ee^{(\tau_4-\tau_{3})\mathcal{L}}f[t](\tau_{3})T^2\ee^{\tau_3\mathcal{L}}u[t](0)\D\tau_3\D\tau_4\right\|\\
&\leq&\int_0^h\left\|\ee^{(h-\tau_4)\mathcal{L}}f[t](\tau_4)\int_0^{\tau_4}\ee^{(\tau_4-\tau_{3})\mathcal{L}}f[t](\tau_{3})T^2\ee^{\tau_3\mathcal{L}}u[t](0)\D\tau_3\right\|\D\tau_4\\
&\leq&C_1\int_0^h\left\|\int_0^{\tau_4}\ee^{(\tau_4-\tau_{3})\mathcal{L}}f[t](\tau_{3})T^2\ee^{\tau_3\mathcal{L}}u[t](0)\D\tau_3\right\|\D\tau_4\\
&\leq&C_1\int_0^h\int_0^{\tau_4}\left\|\ee^{(\tau_4-\tau_{3})\mathcal{L}}f[t](\tau_{3})T^2\ee^{\tau_3\mathcal{L}}u[t](0)\right\|\D\tau_3\D\tau_4\\
&\leq&C_1^2\int_0^h\int_0^{\tau_4}\left\|T^2\ee^{\tau_3\mathcal{L}}u[t](0)\right\|\D\tau_3\D\tau_4,
 \end{eqnarray*}
 where the constant $C_1>0$ depends on the norm of the semigroup operator $\{\ee^{t\mathcal{L}}\}_{t\in [0,t^\star]}$ and the supremum norm of the function $f$.
 Since term $T^2\ee^{\tau_3\mathcal{L}}u[t](0)$ satisfies  $\left\|T^2\ee^{\tau_3\mathcal{L}}u[t](0)\right\|=\mathcal{O}(\omega^{-2})$, we obtain the estimate
 $$\|T^4_t\ee^{h\mathcal{L}}u[t](0)\|\leq C\frac{h^2}{\omega^2}.$$

\noindent Moreover, it can be observed that    expression $\bm{T}_tv[t](h)$, where $\|v[t](h)\|_2\leq 1$ is the solution of the integral equation
\begin{eqnarray*}\label{Series_operator}
    \psi[t](h)=v[t](h)+\int_0^h \ee^{(h-\tau)\mathcal{L}}f[t](\tau)\psi[t](\tau)\D\tau.
\end{eqnarray*}
  By  Gr\"{o}nwall's inequality, expression  $\psi=\bm{T}_tv[t](h)$ is also bounded  in $L^2$ norm for any function $v[t](h)$ such that $\|v[t](h)\|_2\leq 1$.
Using the boundedness of operator $\bm{T}_t$, we can estimate 
\begin{eqnarray*}
\left\|\sum_{d=4}^{\infty}T^d_t\ee^{h\mathcal{L}}u[t](0) \right\| \leq C \min\left\{h^{4},\textcolor{blue}{\frac{h^{2}}{\omega^{2}}},\frac{1}{\omega^4}\right\}.
\end{eqnarray*}
which completes the proof.
\end{proof}
Let us emphasize that the time derivatives of the solution of the highly oscillatory equation (\ref{eq_duhamel}) do not appear in the above estimates,  which means that the constant $C$ is independent of the parameter $\omega$.

By collecting the estimations of integrals presented in Lemmas \ref{lemma1}, \ref{lemma2}, \ref{lemma3}, and estimation of the remainder of the Neumann series in Lemma \ref{lemma4}, one can provide the following local error bound of the scheme.
\begin{theorem}
   Let Assumption \ref{assumption1} be satisfied and let the potential function $f$ be of the form (\ref{positive_frequencies}). Then the local error of the numerical scheme (\ref{scheme}) satisfies the following estimate in the $L^2$ norm
    \begin{eqnarray*}
        \|u(t_0+h)-u^1\| \leq   C\min\left\{h^4, \frac{h^2}{\omega^2}, \frac{1}{\omega^3}\right\},
    \end{eqnarray*}
    where constant $C$    is independent of time step $h$ and parameter $\omega$.
\end{theorem}
\subsection{The  case involving negative frequencies}
The situation becomes more complicated when we perform the error analysis of the proposed numerical integrator for potential function $f$ in the general form (\ref{ideal_function_f}). 
Let $\bm{n}=(n_1,\dots,n_d) \in \bm{N}^d$, where set $\bm{N}^d$ is defined in (\ref{set_N}).
Coordinates of  $\bm{n}$ may satisfy
 $$
 n_j + n_{j-1}+\dots +n_{r+1}+n_r = 0,
 $$
 for certain $1\leq j < r\leq d$, and therefore $\bm{n}$ is orthogonal to the boundary of simplex $\sigma_d(h)$.  Vector $\bm{n}$ does not satisfy the nonresonance condition, and, as a result, simple integration by parts does not yield error estimates similar to those presented in Lemmas \ref{lemma1}, \ref{lemma2}, and \ref{lemma3}.
In this case, we still obtain the fourth-order local error estimate of the numerical scheme
        $\|u(t_0+h)-u^1\| \leq   Ch^4$, where $C$ is independent of $\omega$ and $h$, but we wish to derive a numerical scheme whose accuracy   improves significantly with increasing $\omega$.  

At this stage, we consider two bivariate integrals from the Neumann series,  $ I[F_{\bm{n}_1},\sigma_2(h)]$ and $I[F_{\bm{n}_2},\sigma_2(h)]$, where $\bm{n}_1 = (-n,n)$ and $\bm{n}_2 = (n,-n)$.  Vectors $\bm{n}_1, \bm{n}_2 \in \bm{N}^2$  are orthogonal to the boundary of simplex $\sigma_2(h)$.   By  integration by parts one can show that $I[F_{\bm{n}_1},\sigma_2(h)]\sim \mathcal{O}(\omega^{-1})$, $I[F_{\bm{n}_2},\sigma_2(h)]\sim \mathcal{O}(\omega^{-1})$
but sum of the integrals satisfies $\left(I[F_{\bm{n}_1},\sigma_2(h)] +[F_{\bm{n}_2},\sigma_2(h)]\right) \sim \mathcal{O}(\omega^{-2})$ \cite{perczyński2023asymptotic}. 
We exploit this fact by imposing an additional interpolation condition to construct Filon's quadrature rule for the sum of two bivariate integrals that do not satisfy the nonresonance condition. We also assume that coefficients of function $f$ satisfy $\alpha_{-n}=\alpha_n, \ \forall n\in \bm{N}^1$,  therefore   $F_{\bm{n}_1}= F_{\bm{n}_2}=:F_{\bm{n}}$.

\begin{theorem}
    Let coefficients $\alpha_{-n}=\alpha_n, \ \forall n\in \bm{N}^1$ and consider function $F_{\bm{n}}$ of the form (\ref{function_f_operator}), i.e. $F_{\bm{n}}(\tau_1,\tau_2)= \ee^{(h-\tau_2)\mathcal{L}}\alpha_n[t](\tau_2)\ee^{(\tau_2-\tau_1)\mathcal{L}}\alpha_n[t](\tau_1)\ee^{\tau_1\mathcal{L}}u[t](0)$. Let polynomial 
    $p(\tau_1,\tau_2) = b_0+ b_1\tau_1+b_2\tau_2+ b_3\tau_1\tau_2$
    satisfies the following interpolation conditions
    $$p(0,0) = F_{\bm{n }}(0,0), \quad p(0,h) = F_{\bm{n }}(0,h),\quad p(h,h) = F_{\bm{n }}(h,h),$$
    and
    \begin{eqnarray}\label{new_condition}
    \int_0^h \partial_{\tau_1}^1p(\tau_2,\tau_2)\D\tau_2 =\int_0^h \partial_{\tau_1}^1F_{\bm{n}}(\tau_2,\tau_2)\D\tau_2.
    \end{eqnarray}
    Then 
$$
\left\|\int_0^h\int_0^{\tau_2}(F_{\bm{n}}-p)(\tau_1,\tau_2)\ee^{\ii\omega n( \tau_1-\tau_2)}+(F_{\bm{n}}-p)(\tau_1,\tau_2)\ee^{\ii\omega n(- \tau_1+\tau_2)}\D\tau_1\D\tau_2\right\| \leq   C\min\left\{h^4,\frac{h^2}{\omega^2}, \frac{1}{\omega^3}\right\}.$$
\end{theorem}
\begin{proof}
    For simplicity, we can assume $n=1$. 
    It follows from the previous considerations that
    $(F_{\bm{n}}-p) = \mathcal{O}(h^2)$, $\partial_{\tau_1}^1(F_{\bm{n}}-p) = \mathcal{O}(h)$ and $\partial_{\tau_2}^1(F_{\bm{n}}-p) = \mathcal{O}(h)$.  
    Integration by parts and application of interpolation conditions gives 
    \begin{eqnarray*}
&&\left\|\int_0^h\int_0^{\tau_2}(F_{\bm{n}}-p)(\tau_1,\tau_2)\ee^{\ii\omega(\tau_1-\tau_2)}+(F_{\bm{n}}-p)(\tau_1,\tau_2)\ee^{\ii\omega(-\tau_1+\tau_2)}\D\tau_1\D\tau_2\right\|\leq\\
&&\frac{1}{\omega}\left(\left\|\int_0^h(F_{\bm{n}}-p)(\tau_2,\tau_2)-(F_{\bm{n}}-p)(0,\tau_2)\ee^{-\ii\omega\tau_2}\D\tau_2-\int_0^h(F_{\bm{n}}-p)(\tau_2,\tau_2)-(F_{\bm{n}}-p)(0,\tau_2)\ee^{\ii\omega\tau_2}\D\tau_2\right\|\right)+\\
&&\frac{1}{\omega^2}\left\|\int_0^h \partial_{\tau_1}^1((F_{\bm{n}}-p)(\tau_2,\tau_2))\D\tau_2\right\|+\frac{1}{\omega^2}\left\|\int_0^h\partial_{\tau_1}^1((F_{\bm{n}}-p)(0,\tau_2))\ee^{-\ii\omega\tau_2}\D\tau_2\right\|+\\
&& \frac{1}{\omega^2}\left\|\int_0^h \partial_{\tau_1}^1((F_{\bm{n}}-p)(\tau_2,\tau_2))\D\tau_2\right\|+\frac{1}{\omega^2}\left\|\int_0^h\partial_{\tau_1}^1((F_{\bm{n}}-p)(0,\tau_2))\ee^{\ii\omega\tau_2}\D\tau_2\right\|+\\
&&\frac{1}{\omega^2}\left\|\int_0^h\int_0^{\tau_2}\partial_{\tau_1}^2(F_{\bm{n}}-p)(\tau_1,\tau_2)\ee^{\ii\omega(\tau_1-\tau_2)}\D\tau_1\D\tau_2\right\|+\frac{1}{\omega^2}\left\|\int_0^h\int_0^{\tau_2}\partial_{\tau_1}^2(F_{\bm{n}}-p)(\tau_1,\tau_2)\ee^{\ii\omega(-\tau_1+\tau_2)}\D\tau_1\D\tau_2\right\|\leq\\
&&\frac{1}{\omega^2}\left\|\int_0^h \partial_{\tau_2}^1(F_{\bm{n}}-p)(0,\tau_2)\ee^{-\ii\omega\tau_2}\D\tau_2\right\|+\frac{1}{\omega^2}\left\|\int_0^h \partial_{\tau_2}^1(F_{\bm{n}}-p)(0,\tau_2)\ee^{\ii\omega\tau_2}\D\tau_2\right\|+C\min\left\{\frac{h^2}{\omega^2},\frac{1}{\omega^3}\right\}\leq\\
&&C\min\left\{\frac{h^2}{\omega^2},\frac{1}{\omega^3}\right\},
\end{eqnarray*}
which completes the proof.
\end{proof}
Since the function $F_{\bm{n}}$ is non-oscillatory, we can compute the integral (\ref{new_condition}) efficiently and effortlessly, using methods such as Gauss-Legendre quadrature.

In the case when function $f$ is of the form (\ref{ideal_function_f}), and the coefficients of  $f$ satisfy $\alpha_{-n}=\alpha_n, \ \forall n\in \bm{N}^1$, the improved scheme reads
\begin{eqnarray*}\label{scheme_resonnce}
    u^{k+1} &=& \Bigg( \ee^{h\mathcal{L}}+ \sum_{n_1}  \int_0^h \left(a_{1,0}+a_{1,1}\tau+ a_{1,2}\tau^2+a_{1,3}\tau^3 \right)\ee^{n_1\ii\omega( \tau+t_k)}\D\tau \nonumber \\ &+& \sum_{n_1+n_2\neq 0}\int_{\sigma_2(h)}(a_{2,0}+a_{2,1}\tau_1+a_{2,2}\tau_2)\ee^{\ii\omega(n_1(\tau_1+t_k)+n_2(\tau_2+t_k))}\D\tau_1\D\tau_2\\
    &+& \sum_{n=1 \nonumber }^{N}\int_{\sigma_2(h)}(b_0+ b_1\tau_1+b_2\tau_2+ b_3\tau_1\tau_2)(\ee^{\ii\omega n(\tau_1-\tau_2)}+\ee^{\ii\omega n(-\tau_1+\tau_2)})\D\tau_1\D\tau_2\\ &+&\sum_{n_1,n_2,n_3}\int_{\sigma_3(h)} (a_{3,0}+a_{3,1}\tau_1+a_{3,2}\tau_2+a_{3,3}\tau_3)\ee^{\ii\omega(n_1(\tau_1+t_k)+n_2(\tau_2+t_k)+n_3(\tau_3+t_k))}\D\tau_1\D\tau_2\D\tau_3\Bigg)u^k\nonumber,\\
    t_{k+1} &= &t_k+h \nonumber.
\end{eqnarray*}
\section{Application of the method to the wave equation}\label{sec4}
The proposed numerical scheme can be successfully applied to partial differential equations with the second-time derivative. Consider the equation
\begin{align} \label{main_equation}
&\partial_{tt}u=\mathcal{L} u(x,t)+f(x,t)u(x,t),\qquad t\in [0,t^\star],\ x\in\Omega\subset\mathbb{R}^m,\\ \nonumber
&u(x,0)=u_1(x), \quad \partial_t u(x,0) = u_2(x), \\ \nonumber
& u=0 \ \text{on} \   \partial\Omega \times [0,t^\star],
\end{align}
with function $f$ given in (\ref{ideal_function_f}). We write (\ref{main_equation}) as a first-order system
\arraycolsep=4pt\def\arraystretch{1.2}
\begin{equation}\label{eq:matrix}
\partial_t\left[  \begin{array}{c} \nonumber
u \\
v
\end{array} \right] = \left[  \begin{array}{cc}
0 & \mathcal{I} \\
\mathcal{L} & 0
\end{array} \right]	\left[  \begin{array}{c}
u \\
v
\end{array} \right]+ f\left[  \begin{array}{cc}
0 & 0 \\
\mathcal{I} & 0
\end{array} \right]\left[  \begin{array}{c}
u \\
v
\end{array} \right],
\end{equation}
where   $v  = \partial_t u$. Thus
\begin{equation}\label{eq:matrix_highly_oscillatory}
\partial_t\underbrace{\left[  \begin{array}{c} \nonumber
	u \\ 
    v
	\end{array} \right]}_{\varphi} = \underbrace{\left[  \begin{array}{cc}
	0 & \mathcal{I} \\
	\mathcal{L} & 0
	\end{array} \right]	}_{A}\left[  \begin{array}{c}
u \\
v
\end{array} \right]+\sum_{n=1}^{N}\ee^{\ii n\omega t}\underbrace{ \left[  \begin{array}{cc}
	0 & 0 \\
	\alpha_n & 0
	\end{array} \right] }_{\beta_n}\left[  \begin{array}{c}
u \\
v
\end{array} \right]
\end{equation}
and therefore
\begin{equation}\label{first_order}
\partial_t\varphi = A\varphi+h\varphi, \quad \varphi(x,0)=[u_1(x),u_2(x)], \quad A(u,v) = (v,\mathcal{L}u), \quad \beta_n(u,v)=(0,u\alpha_n),
\end{equation}
where $h(x,t)= \sum_{n=1}^{N}\ee^{\ii n\omega t}\beta_n(x,t)$ is a highly oscillatory function  and $\varphi$ is a vector valued function.
Suppose that $\mathcal{L}$ is a second-order differential operator which has symmetric form
$$\mathcal{L}u = \sum_{i,j=0}^{m}\partial_{x_j}\left(a_{ij}\partial_{x_j}u\right)-cu,$$
where functions $a_{ij}=a_{ji}, \ i,j=1,\dots,m$ and $c\geq 0$.  
By applying  Duhamel's formula we write (\ref{first_order}) as

\begin{eqnarray}\label{eq:Duhamel2}
\varphi(t) = \ee^{tA}\varphi_0+\int_{0}^{t}\ee^{(t-\tau)A}h(\tau)\varphi(\tau)\D\tau.
\end{eqnarray}
 Operator  $A: D(A):= \left[ H^{2}(\Omega)\cap H_0^1(\Omega) \right]\times H_0^1(\Omega)\rightarrow H_0^1(\Omega)\times L^2(\Omega)$  is the infinitesimal generator of  a strongly continuous semigroup  $\{\ee^{tA}\}$ on $H_0^1(\Omega)\times L^2(\Omega)$ \cite{EVANS}.  One can show that the Neumann series converges absolutely and uniformly in the norm of space $H_0^1(\Omega)\times L^2(\Omega)$ to the solution of equation (\ref{eq:Duhamel2}) \cite{perczyński2023asymptotic}.

 \section{Numerical examples}\label{sec5}
In this chapter, we employ the proposed numerical integrator to solve highly oscillatory heat equations and wave equations. For each equation, it is possible to determine the analytical solution,  enabling accurate comparisons with the numerical approximation. The $L^2$ norm of the error is considered in any presented example. In our numerical experiments, to find an approximate solution, we use the Fourier and Chebyshev spectral methods, as described in \cite{trefethen,shen_spectral}. In Examples 1, 3 and 4 we used $M=100$ spatial grid points. Example 2 concerns a two-dimensional case, in which we used $M=20$ grid points.
\\

\noindent \emph{Example 1. The heat equation.}\\
Consider the equation
\begin{align}\label{heat_example_t} \nonumber
&\partial_{t}u= \partial_{xx}^2 u+f(x,t)u(x,t),\qquad t\in [0,1],\ x\in(0,2\pi),\\ 
&u(x,0)=u_0(x), \\
&u(0,t) = 0=u(2\pi,t) ,\nonumber
\end{align}
with initial condition $u_0$
$$u_0(x)=\sin(x),$$
and function $f$
$$f(x,t) = 1-\underbrace{\frac{(-\ii+t(\omega-3\ii))\cos(x)}{\omega}}_{\alpha_1}\ee^{\ii\omega t}+\underbrace{\frac{\sin(x)^2 t^2}{\omega^2}}_{\alpha_2}\ee^{2\ii\omega t}.$$
The potential function $f$ involves time-dependent coefficients $\alpha_1$ and $\alpha_2$
The solution to (\ref{heat_example_t}) is  $$u(x,t) =  \ee^{\ii\ee^{\ii\omega t}\cos(x)t/\omega}\sin(x).$$
Figure \ref{figure1} displays the error of the method for equation (\ref{heat_example_t}).
\begin{figure}[H] 
\captionsetup[subfigure]{labelformat=empty}
\begin{subfigure}{.5\textwidth}
\centering
\includegraphics[height=6.65cm]{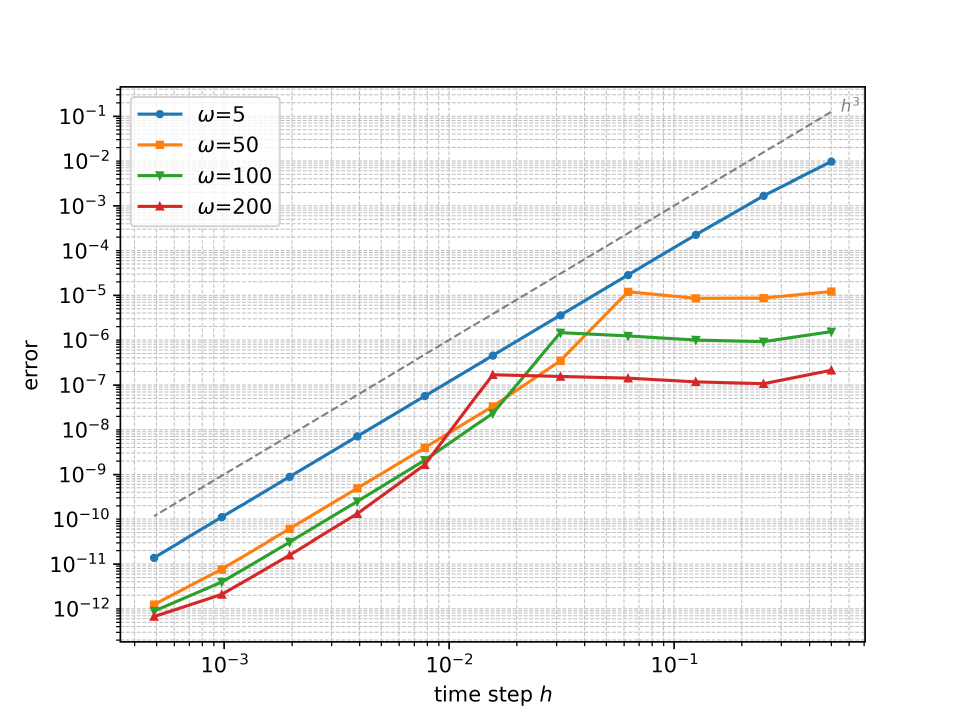}
\caption{}
\end{subfigure}%
\begin{subfigure}{.5\textwidth}
\centering
\includegraphics[height=6.65cm]{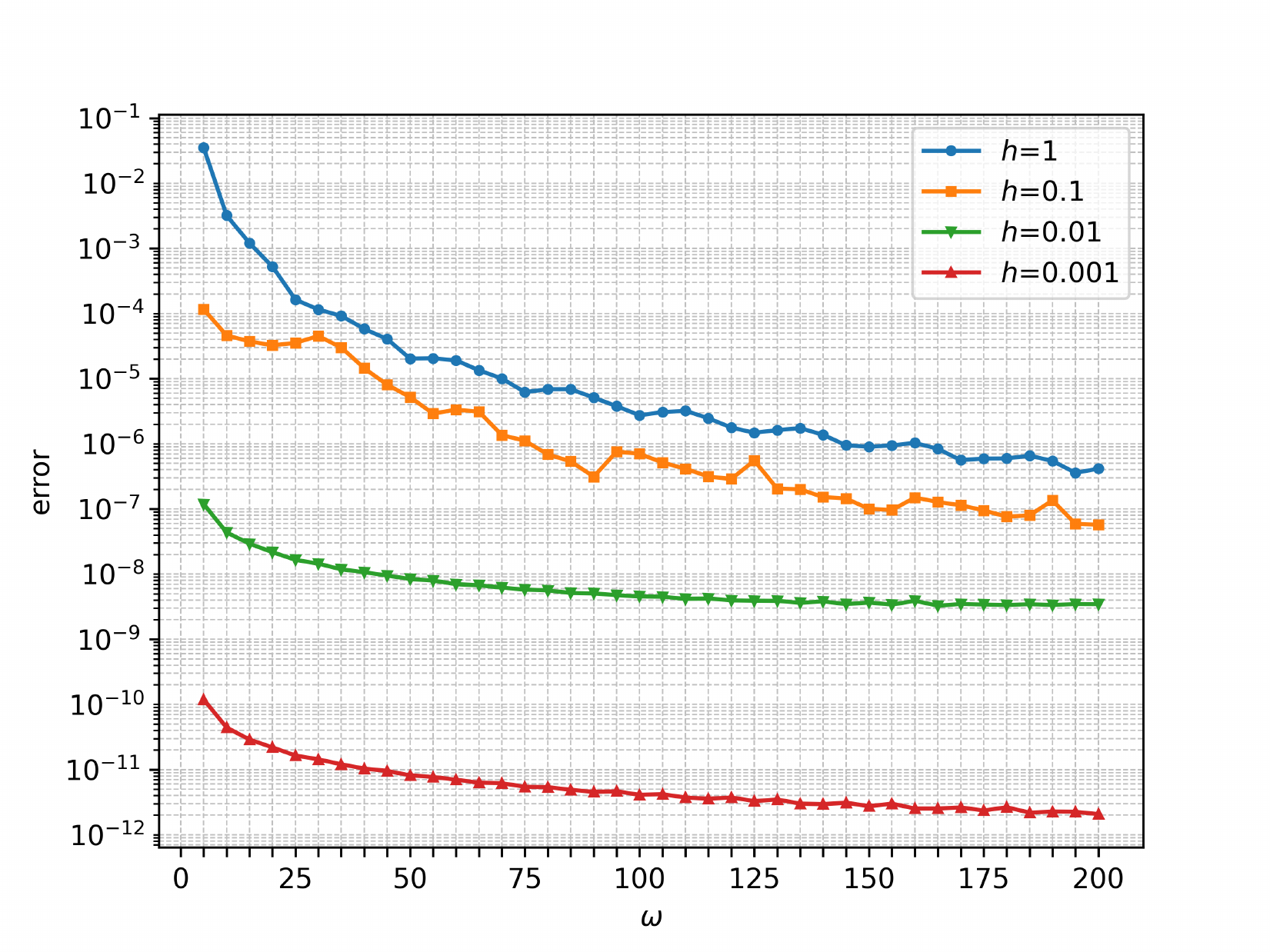}
\caption{} 
\end{subfigure}%
\caption{ Numerical approximation of the solution to equation (\ref{heat_example_t}).
 Error versus time step (left graph) and error versus parameter $\omega$ (right graph).} \label{figure1}
\end{figure} 
\noindent \emph{Example 2. Two-dimensional heat equation.}
\begin{align} \label{2dheat_example} \nonumber
&\partial_{t}u(x,y,t)= \partial_{xx}^2 u(x,y,t)+\partial_{yy}^2 u(x,y,t)+f(x,y,t)u(x,y,t),\qquad t\in [0,1],\ x\in \Omega,\\ 
&u(x,y,0)=u_0(x,y), \\
&u(x,y,t) = 0 \quad \text{on} \ \partial \Omega ,\nonumber
\end{align}
where domain 
$\Omega = [-1,1] \times [-1,1]$. 
The initial condition is $$u_0(x,y) = \sin(\pi x)\sin(\pi y)\ee^{\cos(\pi x)\cos(\pi y)/\omega} ,$$
and function $f$
$$f(x,y,t) = 2\pi^2+ \underbrace{\frac{(6\pi^2+\ii\omega)\cos(\pi x)\cos(\pi y)}{\omega}}_{\alpha_1}\ee^{\ii\omega t}+ \underbrace{\frac{0.5\pi^2(-1+\cos(2\pi x)\cos(2\pi y))}{\omega^2}}_{\alpha_2}\ee^{2\ii\omega t}.$$
The solution to equation (\ref{2dheat_example}) reads $$u(x,y,t) = \sin(\pi x)\sin(\pi y)\exp(\exp(\ii\omega t)\cos(\pi x)\cos(\pi y)/\omega).$$
Figure \ref{figure2} presents the error of the proposed method applied to equation (\ref{2dheat_example}).
\begin{figure}[H]
\captionsetup[subfigure]{labelformat=empty}
\begin{subfigure}{.5\textwidth}
\centering
\includegraphics[height=6.65cm]{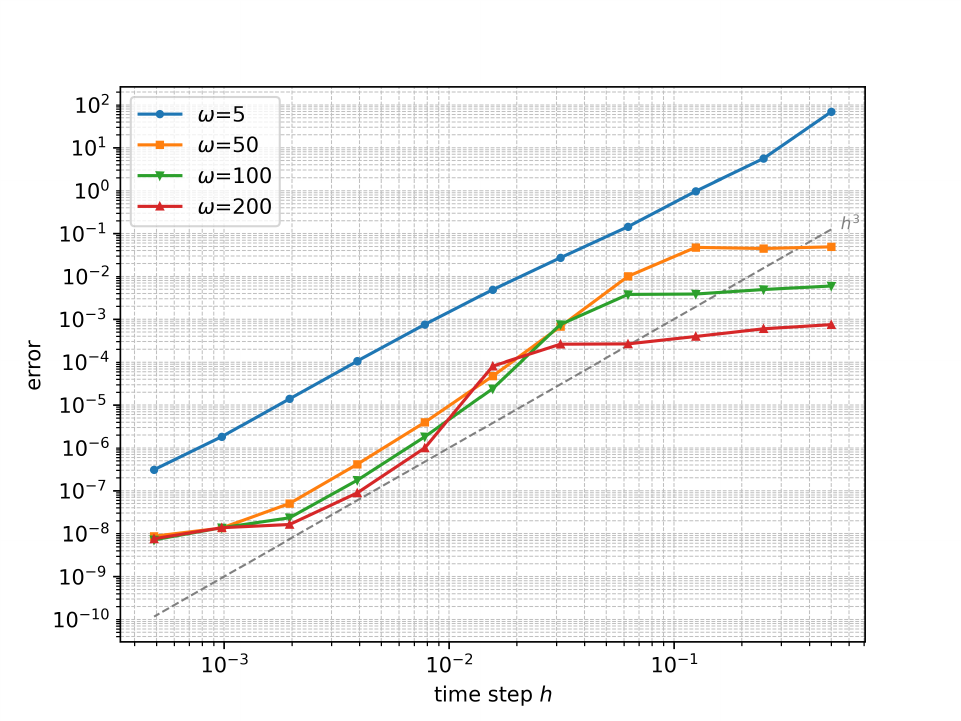}
\caption{}
\end{subfigure}%
\begin{subfigure}{.5\textwidth}
\centering
\includegraphics[height=6.65cm]{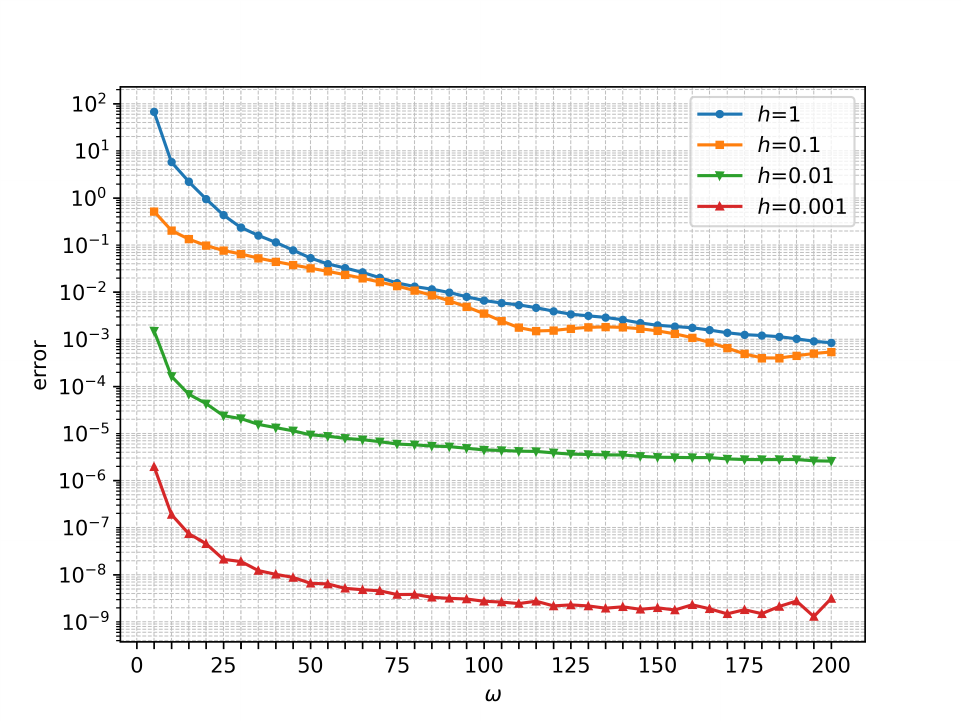}
\caption{}
\end{subfigure}%
\caption{ Numerical approximation of the solution to equation (\ref{2dheat_example}).
Error versus time step (left graph) and error versus parameter $\omega$ (right graph).} \label{figure2}
\end{figure}

\noindent \emph{Example 3. The wave equation with nonresonance points.}\\
\begin{align} \label{wave_example_bz}
 &\partial_{tt}u=\partial_{xx}  u+f(x,t)u(x,t),\qquad t\in [0,1],\ x\in (-L,L), \quad L=8,\\  \nonumber
 &u(x,0)=u_0(x), \quad \partial_t u(x,0) = v_0(x),\\
 &u(-L,t) =u(L,t) ,\nonumber \\
 & \partial_{t}u(-L,t) =\partial_{t}u(L,t),   \nonumber
 \end{align}
where initial conditions $$u_0(x)= \ee^{-x^2(1/2+1/\omega^2)}, \quad v_0(x) = -\frac{ \ii x^2}{\omega} u_0(x), $$
and function $f$
$$f(x,t) = 1-x^2+\frac{2+x^2(-4+\omega^2)}{\omega^2}\ee^{\ii\omega t}-\frac{x^2(4+x^2\omega^2)}{\omega^4}\ee^{2\ii\omega t}. $$
Solution to (\ref{wave_example_bz}) is equal
to $$u(x,t) = \ee^{-x^2/2}\ee^{-\ee^{\ii t \omega}x^2/\omega^2}.$$
Figure \ref{figure3} illustrates the error associated with the approximation of the solution to equation (\ref{wave_example_bz}).
\begin{figure}[H]
\captionsetup[subfigure]{labelformat=empty}
\begin{subfigure}{.5\textwidth}
\centering
\includegraphics[height=6.65cm]{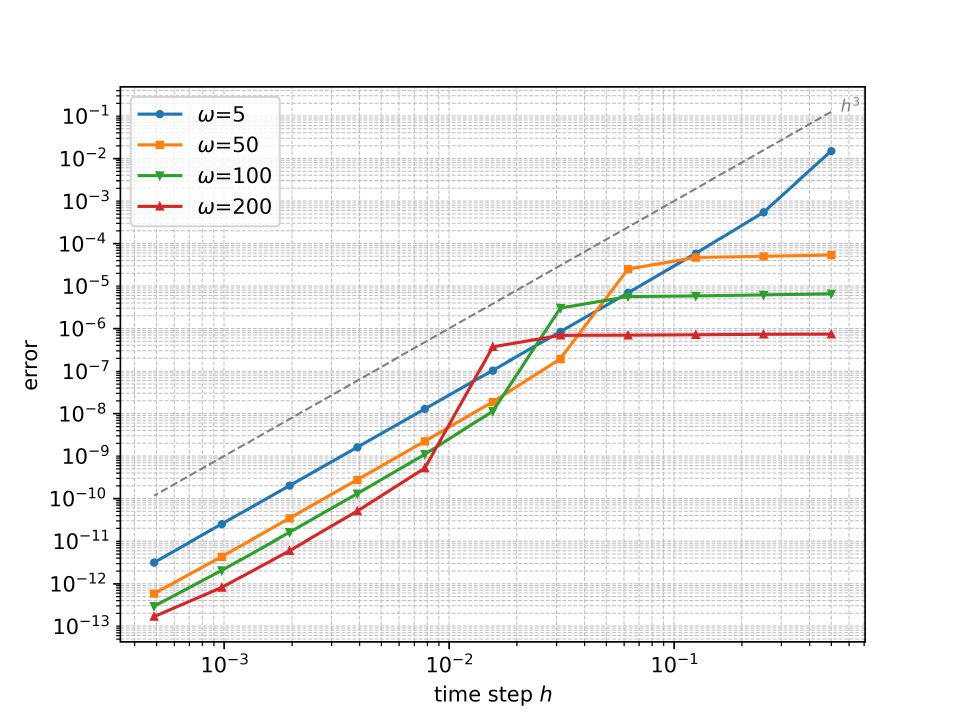}
\caption{}
\end{subfigure}%
\begin{subfigure}{.5\textwidth}
\centering
\includegraphics[height=6.65cm]{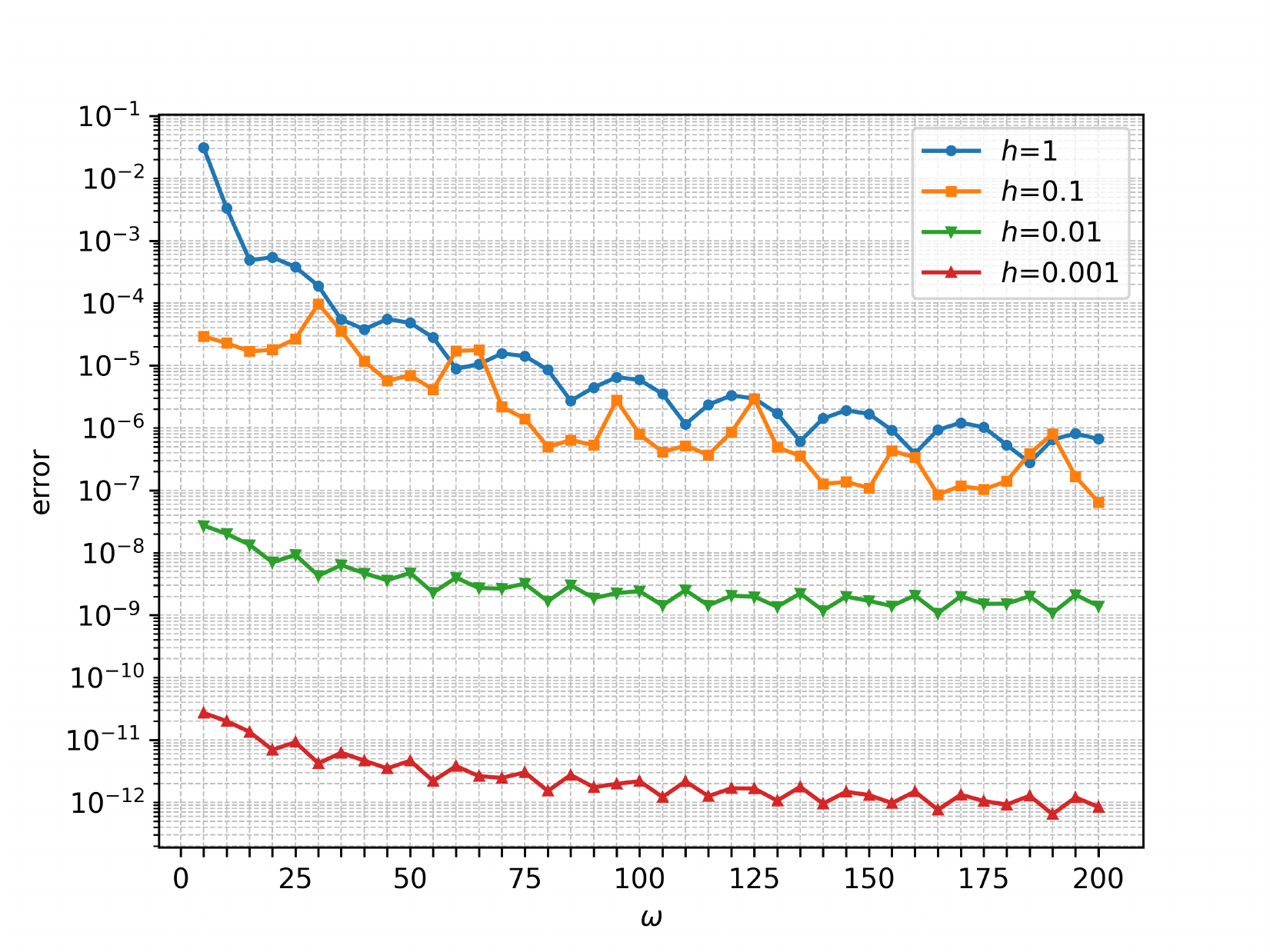}
\caption{}
\end{subfigure}%
\caption{ Numerical approximation of the solution to equation (\ref{wave_example_bz}).
Error versus time step (left graph) and error versus  $\omega$ (right graph).}\label{figure3}
\end{figure}

\noindent \emph{Example 4. The wave equation with resonance points.}\\
In the last example, consider now the wave equation with potential function $f$ with negative frequencies
 \begin{align} \label{wave_example}
 &\partial_{tt}u=\partial_{xx} \nonumber u+f(x,t)u(x,t),\qquad t\in [0,1],\ x\in (-L,L), \quad L=8,\\ 
 &u(x,0)=\ee^{-x^2(1/2+1/\omega^2)}, \quad \partial_t u(x,0) = 0,\\
 &u(-L,t) =u(L,t) ,\nonumber \\
 & \partial_{t}u(-L,t) =\partial_{t}u(L,t),   \nonumber
 \end{align}
 where function $f$ takes the form
 $$f(x,t) = 1-x^2+\frac{(2+x^2\omega^2-4x^2)\cos(\omega t)}{\omega^2}-\frac{4x^2\cos^2(\omega t)}{\omega^4}+\frac{x^4\sin^2(\omega t)}{\omega^2}.$$
  The solution of (\ref{wave_example}) is equal to $$u(x,t) = \ee^{-\cos(\omega t)x^2/\omega^2}\ee^{-x^2/2}.$$
  Figure \ref{figure4} presents the error of the proposed method for equation (\ref{wave_example}).
\begin{figure}[H]
\captionsetup[subfigure]{labelformat=empty}
\begin{subfigure}{.5\textwidth}
\centering
\includegraphics[height=6.65cm]{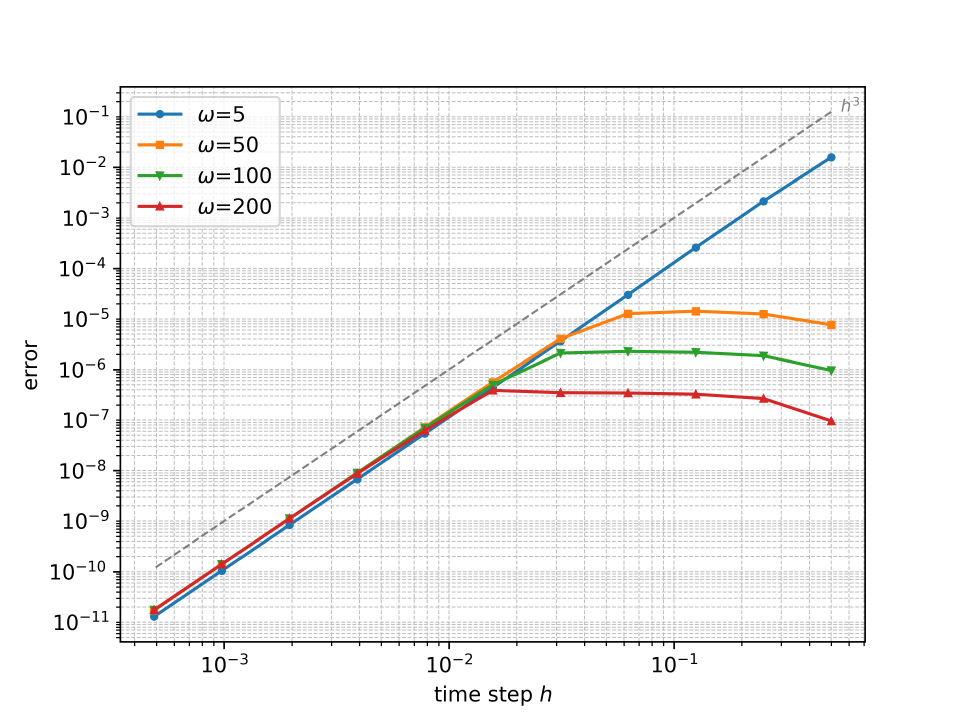}
\caption{}
\end{subfigure}%
\begin{subfigure}{.5\textwidth}
\centering
\includegraphics[height=6.65cm]{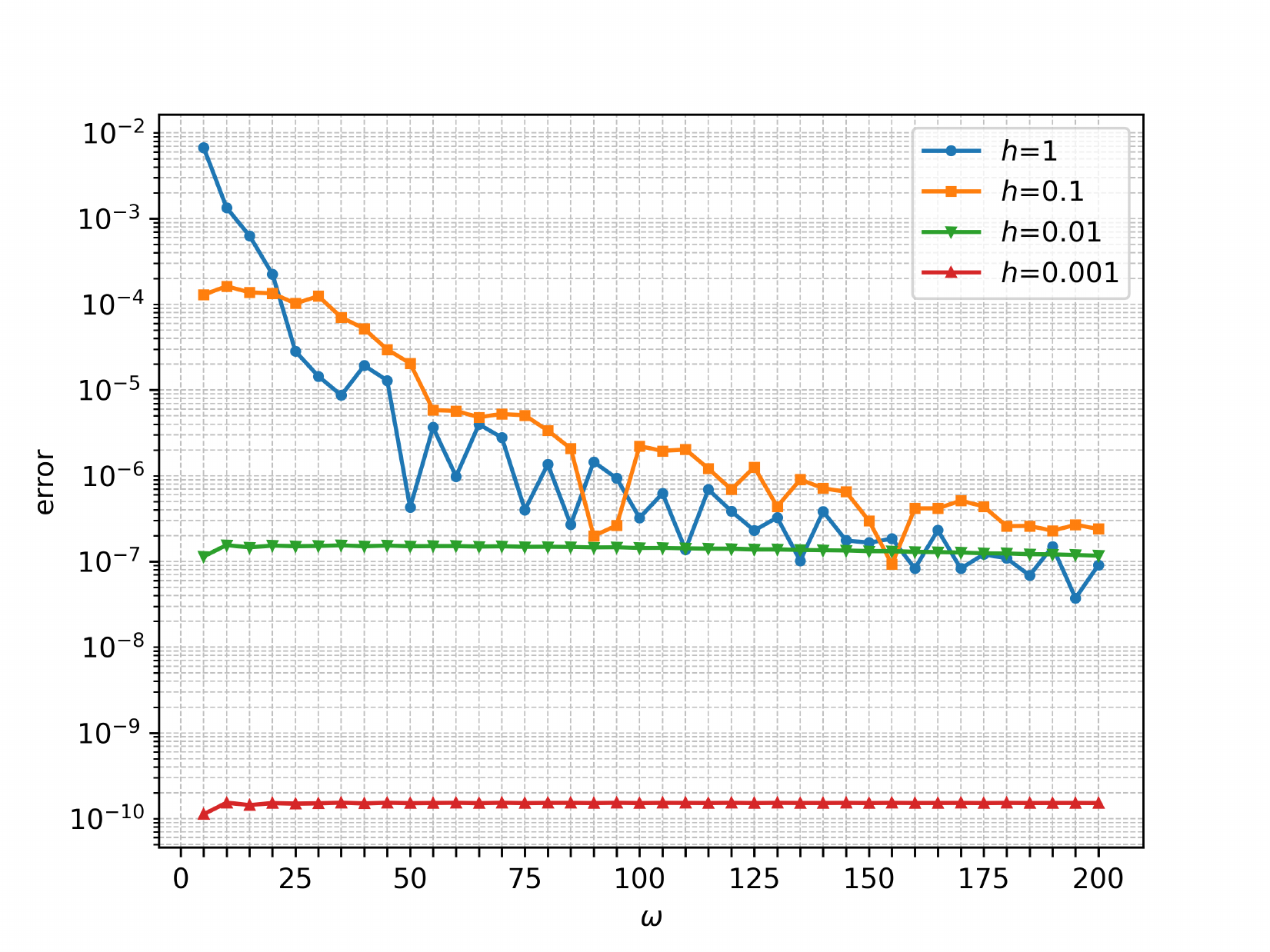}
\caption{}
\end{subfigure}%
\caption{ Numerical approximation of the solution to equation (\ref{wave_example}).
Error versus time step (left figure) and error versus $\omega$ (right figure).}\label{figure4}
\end{figure}

\noindent The Neumann series converges for all variables $t$, unlike the Magnus expansion, which converges only locally. Therefore, in the proposed scheme, any time step can be taken to find an approximate solution. In Figure \ref{figure_bigomega}, we illustrate the error of the method for all four examples with step size $h=1$, where $\omega$ ranges from 5 to 1000. In each graph of error versus time step $h$, it can be observed that the proposed method is effective for both small ($\omega=5$) and large ($\omega=200$) oscillatory parameter $\omega$.

\begin{figure}
\begin{center}
	\includegraphics[width=0.7\textwidth]{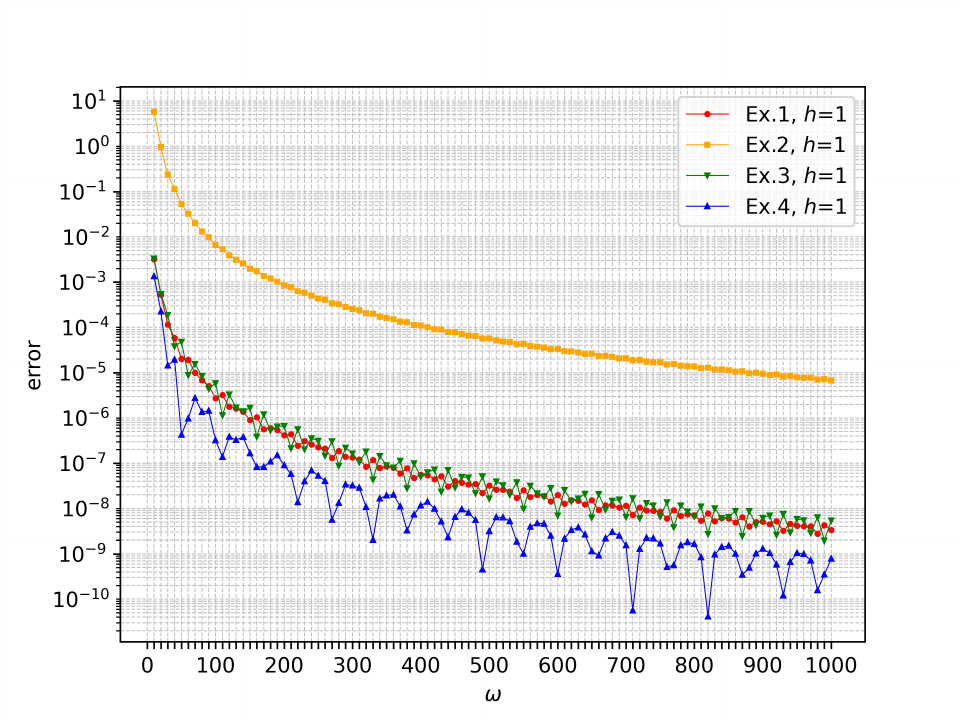}  
	\caption{The error of approximating the solutions of equations (\ref{heat_example_t}), (\ref{2dheat_example}), (\ref{wave_example_bz}), and (\ref{wave_example}),  for $\omega$ ranging from 5 to 1000, with step sizes $h=1$.} \label{figure_bigomega}
	\end{center}
\end{figure}

\subsection{Comparison with other methods}
In this section, we compare the proposed numerical method (denoted as $NF3$) with selected existing methods. For this purpose, we used schemes based on the Magnus expansion: the exponential fourth-order method (denoted as $M4$) and the exponential midpoint method of order two (denoted as $M2$). Both integrators are described in detail in \cite{hochbruck_ostermann_2010}. Each scheme was applied to equations (\ref{heat_example_t}), (\ref{2dheat_example}), and (\ref{wave_example_bz}), where in each case the parameter $\omega=500$.
As is well known, the methods $M4$ and $M2$ are very effective for nonoscillatory equations. However, for a large parameter $\omega$ which accounts for the oscillation of the equation, their effectiveness is limited. The proposed $NF3$ method performs particularly well in a highly oscillatory regime.
The results of the comparisons are shown in Figure \ref{fig:wykresy}. 

   \begin{figure}[H]
    \centering
    \begin{subfigure}{0.33\textwidth}
        \centering
        \includegraphics[width=\linewidth,height=6cm]{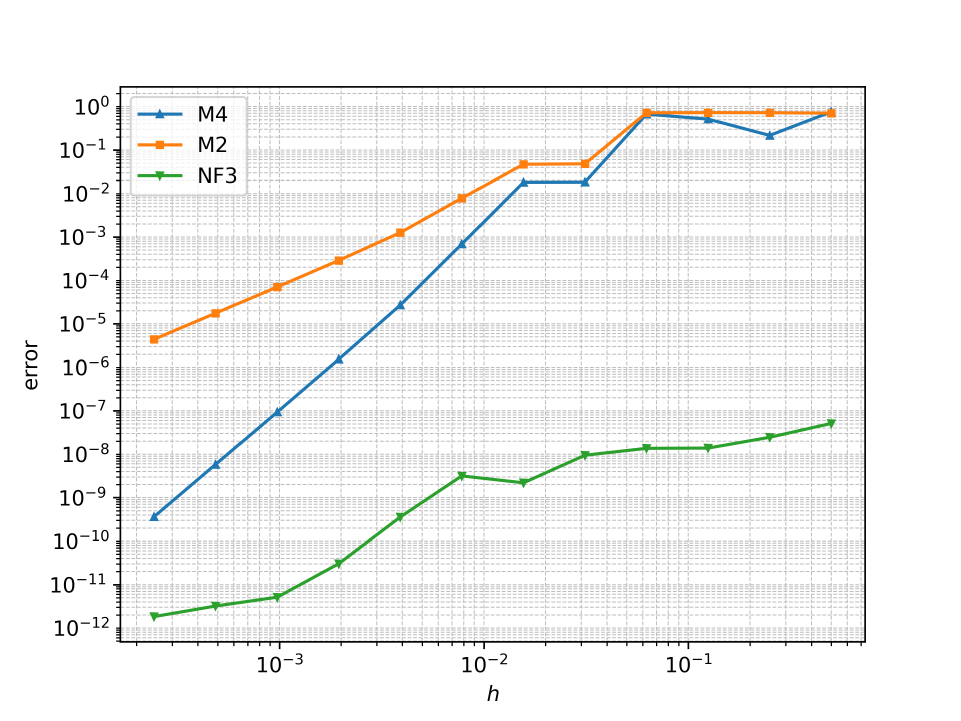}
        \caption{\emph{Example 1}}
    \end{subfigure}%
    \hfill
    \begin{subfigure}{0.33\textwidth}
        \centering
        \includegraphics[width=\linewidth,height=6cm]{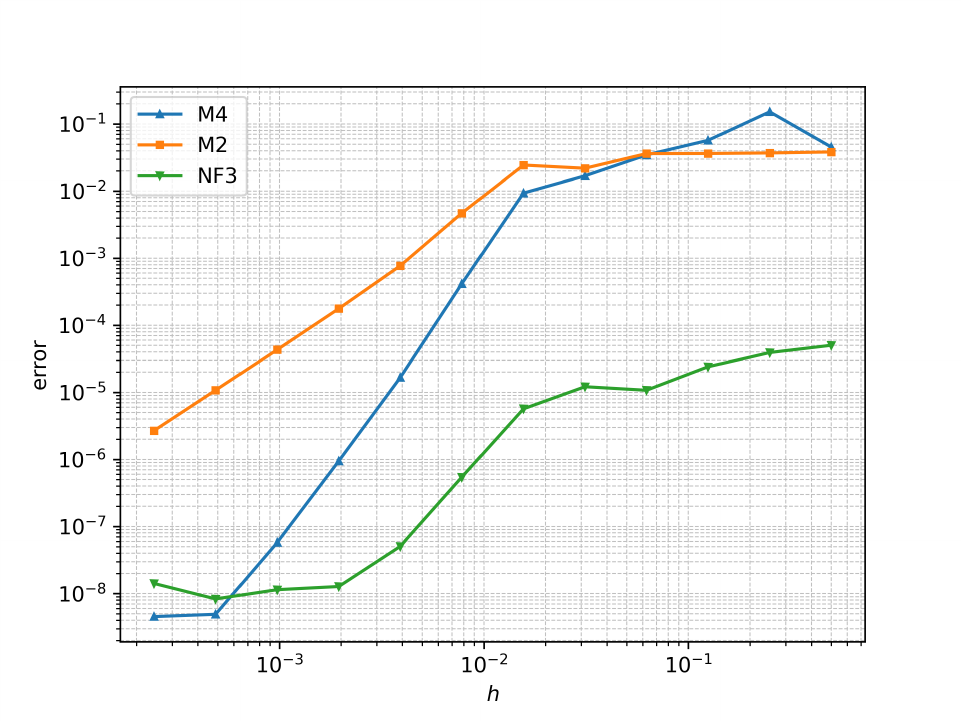}
        \caption{\emph{Example 2}}
    \end{subfigure}%
    \hfill
    \begin{subfigure}{0.33\textwidth}
        \centering
        \includegraphics[width=\linewidth,height=6cm]{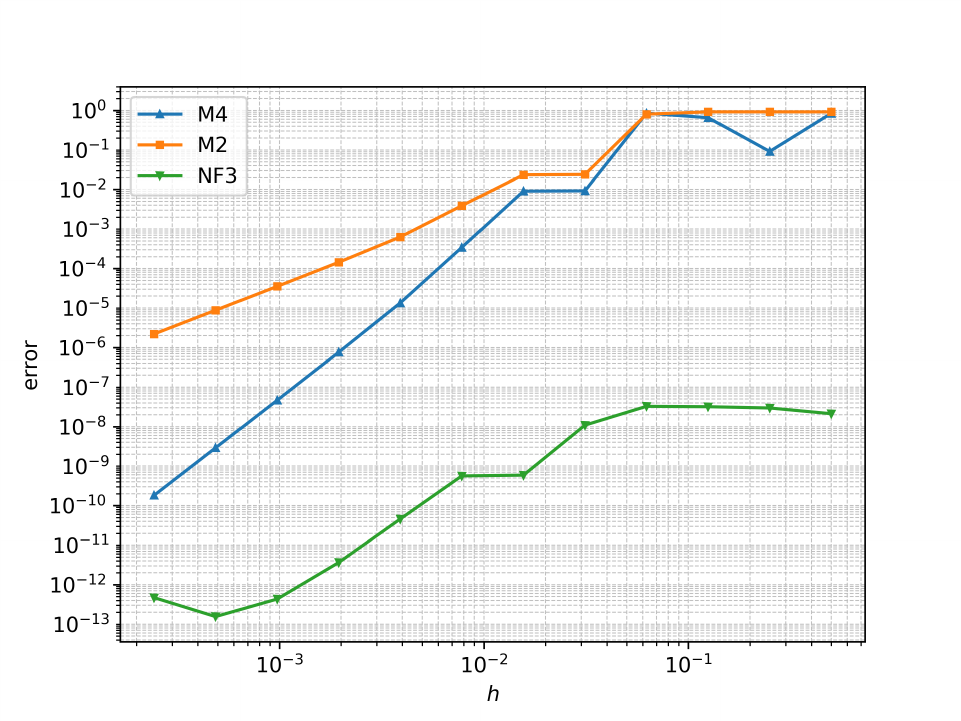}
        \caption{\emph{Example 3}}
    \end{subfigure}
    
    \medskip
    
    \begin{subfigure}{0.33\textwidth}
        \centering
        \includegraphics[width=\linewidth,height=6cm]{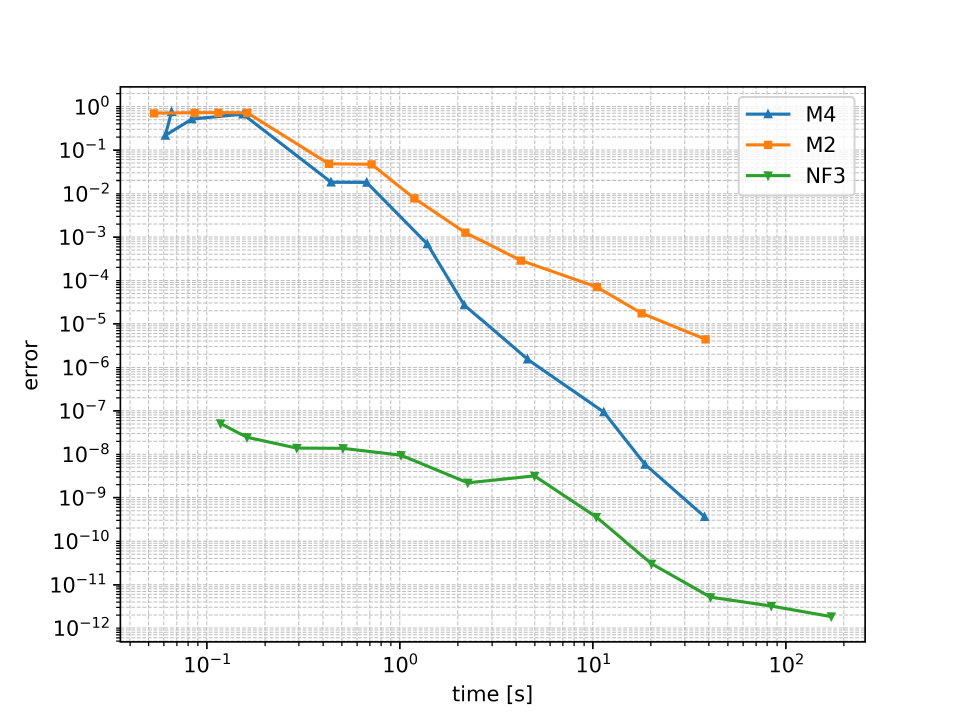}
        \caption{\emph{Example 1}}
    \end{subfigure}%
    \hfill
    \begin{subfigure}{0.33\textwidth}
        \centering
        \includegraphics[width=\linewidth,height=6cm]{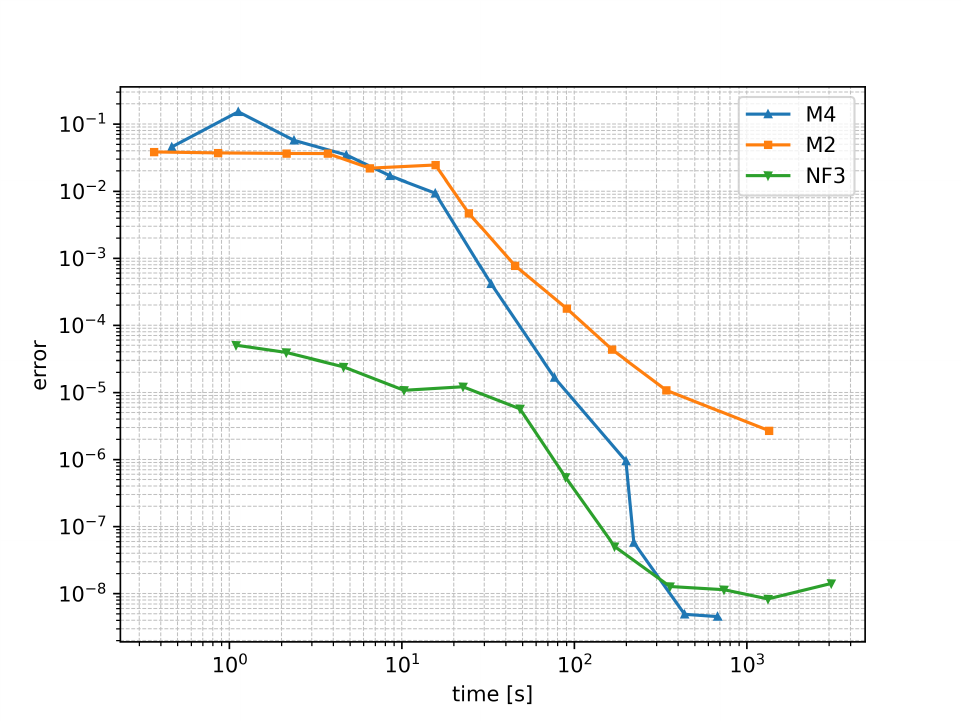}
        \caption{\emph{Example 2}}
    \end{subfigure}%
    \hfill
    \begin{subfigure}{0.33\textwidth}
        \centering
        \includegraphics[width=\linewidth,height=6cm]{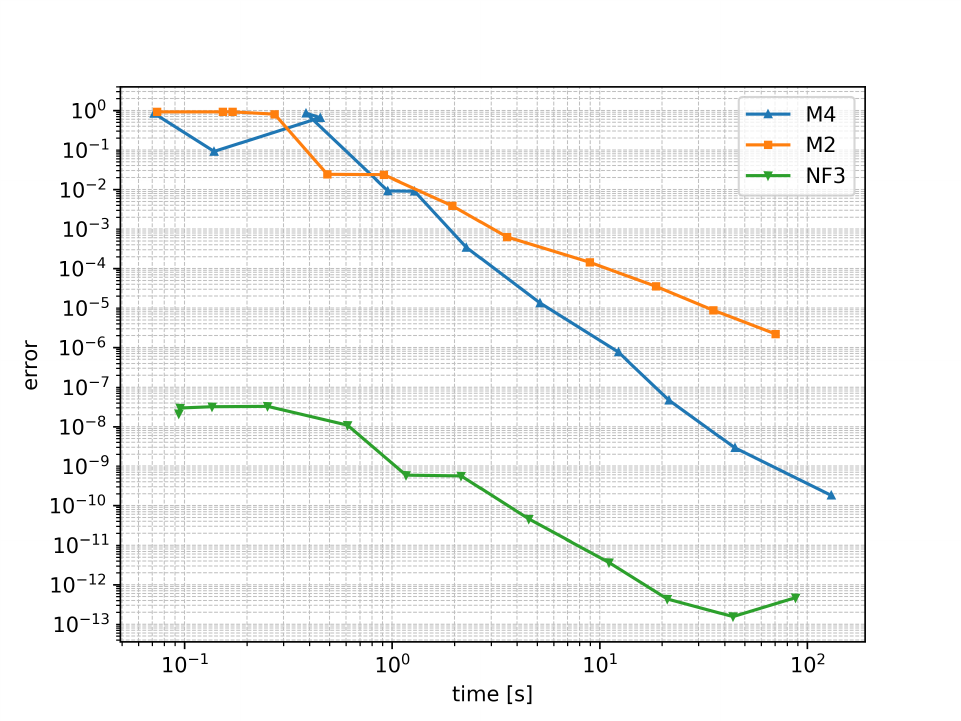}
        \caption{\emph{Example 3}}
    \end{subfigure}
    
    \caption{Comparison of the proposed method $NF3$ with the exponential 4th order method (M4) and the exponential 2nd order midpoint method (M2). The numerical schemes have been applied to the equations (\ref{heat_example_t}), (\ref{2dheat_example}) and (\ref{wave_example_bz}), where $\omega=500$. Top row presents accuracy of schemes and the bottom row time of computation in seconds.}
    \label{fig:wykresy}
\end{figure}

	\section{Conclusion}
	The paper proposed a numerical integrator designed for linear partial differential equations with a highly oscillatory potential function of type (\ref{ideal_function_f}).   
	The numerical scheme is constructed by expanding the solution into the Neumann series. Subsequently, the first three integrals of the Neumann series are approximated using the Filon method. The classical and asymptotic order of the scheme is 3, which is confirmed by numerical experiments. The method is effective for both small and large values of the oscillatory parameter $\omega$.
	\\
	There are possible modifications of the proposed integrator.  For integrals from the Neumann series, different extensions of the Filon method can be applied. For example, the nonoscillatory integrands can be approximated not only at points that are the ends of the integration interval but also at intermediate points. This should further improve the accuracy of the method.

\begin{appendices}

\section{}\label{secA1}
In this section, we provide precise calculations for approximating highly oscillatory integrals: univariate, bivariate, and trivariate, from the Neumann series. 
\\
  
\noindent \emph{Univariate integral}

\noindent Consider  the following  integral
\begin{eqnarray*}
    \int_0^h \ee^{(h-\tau)\mathcal{L}}\alpha_{n_1}[t](\tau)\ee^{\tau\mathcal{L}}u[t](0)\ee^{n_1\ii\omega\tau}\D\tau.
\end{eqnarray*}
We denote $F(\tau):=\ee^{(h-\tau)\mathcal{L}}\alpha_{n_1}[t](\tau)\ee^{\tau\mathcal{L}}u[t](0) $.
Function $F$ is approximated  by using Hermite interpolation
$F(0)=p(0)$, $F(h)=p(h)$, $F'(0)=p'(0)$, $F'(h)=p'(h)$.
The polynomial $p$ approximating function  $F$ is 
\begin{eqnarray*}
p(\tau)&=&F(0)+\tau F'(0)+\frac{\tau^2}{h^2}(3F(h)-3F(0)-2hF'(0)-hF'(h))+\frac{\tau^3}{h^3}(hF'(h)-2F(h)+2F(0)+hF'(0)),\\
&=& F(0)+a_{1,1}\tau + a_{1,2}\tau^2 + a_{1,3}\tau^3,
\end{eqnarray*}
where
\begin{eqnarray*}
    a_{1,1}&=&F'(0),\\
    a_{1,2}&=&\frac{1}{h^2}(3F(h)-3F(0)-2hF'(0)-hF'(h)),\\
    a_{1,3}&=& \frac{1}{h^3}(hF'(h)-2F(h)+2F(0)+hF'(0)).
\end{eqnarray*}
If function $\alpha_{n_1}$ is dependent on time, then
\begin{eqnarray*}
&&F(0)= \ee^{h\mathcal{L}}\alpha_{n_1}[t](0)u[t](0), \qquad  F'(0) = \left(-\ee^{h\mathcal{L}}ad^1_\mathcal{L}(\alpha_{n_1}[t](0))+\ee^{h\mathcal{L}}\alpha_{n_1}[t]'(0)\right)u[t](0),\\
&&F(h) =\alpha_{n_1}[t](h)\ee^{h\mathcal{L}}u[t](0), \qquad  F'(h) = \left(-ad^1_\mathcal{L}(\alpha_{n_1}[t](h))\ee^{h\mathcal{L}}+\alpha_{n_1}[t]'(h)\ee^{h\mathcal{L}}\right)u[t](0),
\end{eqnarray*}
where
$ad_\mathcal{L}^{1}\big(\alpha\big)=\big[\mathcal{L}, \alpha\big]$ and $[X,Y]\equiv XY-YX$ is the commutator of $X$ and $Y$.
Approximation of the univariate integral is as follows
\begin{eqnarray*}
    \int_0^hF(\tau)\ee^{n_1\ii\omega \tau}\D\tau\approx  \int_0^h \left(F(0)+a_{1,1}\tau+a_{1,2}\tau^2+a_{1,3}\tau^3\right)\ee^{n_1\ii\omega \tau}\D\tau.
\end{eqnarray*}
\noindent \emph{Bivariate integral with nonresonnce points}\\
\noindent We approximate the following bivariate integral
\begin{eqnarray*}
  \int_0^h \int_0^{\tau_2}F(\tau_1,\tau_2)\ee^{\ii\omega(n_1+n_2)}\D\tau_1\D\tau_2, 
\end{eqnarray*}
where $F(\tau_1,\tau_2) := \ee^{(h-\tau_2)\mathcal{L}}\alpha_{n_2}[t](\tau_2)\ee^{(\tau_2-\tau_1)\mathcal{L}}\alpha_{n_1}[t](\tau_1)\ee^{\tau_1\mathcal{L}}u[t](0)
$.
Function $F$ is interpolated in the nodes $(0,0)$, $(h,h)$, $(0,h)$, and $F(0,0)=p(0,0)$, $F(0,h)=p(0,h)$, $F(h,h)=p(h,h)$,
where $p(\tau_1,\tau_2)$ is a linear polynomial that approximate function $F$
$$F(\tau_1,\tau_2) =   \underbrace{F(0,0)+a_{2,1}\tau_1+a_{2,2}\tau_2}_{p(\tau_1,\tau_2)}+ \mathcal{O}(h^2),$$
where 
\begin{eqnarray*}
   a_{2,1} = \frac{1}{h}(F(h,h)-F(0,h)), \quad a_{2,2} = \frac{1}{h}(F(0,h)-F(0,0)), 
\end{eqnarray*}
and
\begin{eqnarray*}
    F(0,0) &=& \ee^{h\mathcal{L}}\alpha_{n_2}[t](0)\alpha_{n_1}[t](0)u[t](0), \quad F(0,h) = \alpha_{n_2}[t](h)\ee^{h\mathcal{L}}\alpha_{n_1}[t](0)u[t](0), \\
    F(h,h)&=&\alpha_{n_2}[t](h)\alpha_{n_1}[t](h)\ee^{h\mathcal{L}}u[t](0).
\end{eqnarray*}
Approximation of the bivariate integral reads
\begin{eqnarray*}  \int_0^h\int_0^{\tau_2}F(\tau_1,\tau_2)\ee^{\ii\omega(n_1\tau_1+n_2\tau_2)}\D\tau_1\D\tau_2 \approx \int_0^h\int_0^{\tau_2} (F(0,0)+a_{2,1}\tau_1+a_{2,2}\tau_2)\ee^{\ii\omega(n_1\tau_1+n_2\tau_2)}\D\tau_1\D\tau_2.
\end{eqnarray*}
\\

\noindent \emph{Bivariate integrals with resonnce points}

\noindent Consider the following sum of two integrals with resonnce points $(n,-n)$ and $(-n,n)$
    \begin{eqnarray*}  \int_0^h\int_0^{\tau_2}F(\tau_1,\tau_2)\ee^{\ii\omega n(\tau_1-\tau_2)}\D\tau_1\D\tau_2 +\int_0^h\int_0^{\tau_2}F(\tau_1,\tau_2)\ee^{\ii\omega n(-\tau_1+\tau_2)}\D\tau_1\D\tau_2,
\end{eqnarray*}
where $F(\tau_1,\tau_2):= \ee^{(h-\tau_2)\mathcal{L}}\alpha_n[t](\tau_2)\ee^{(\tau_2-\tau_1)\mathcal{L}}\alpha_n[t](\tau_1)\ee^{\tau_1\mathcal{L}}u[t](0)$. 
The bivariate integrals with resonance points necessitate the imposition of an additional interpolating condition.
Let 
    $p(\tau_1,\tau_2) = F(0,0)+ b_1\tau_1+b_2\tau_2+ b_3\tau_1\tau_2$, be a polynomial with
coefficients $b_j$ defined by the formulas
\begin{eqnarray*}
    b_1 &=& \frac{1}{h}(F(0,h)-F(0,0)),\\
    b_2&=&\frac{1}{h}(2X+F(0,h)-F(h,h)),\\
    b_3&=&\frac{2}{h^2}(F(h,h)-F(0,h)-X),
\end{eqnarray*}
where $$X = \int_0^h\partial_{\tau_1}^1 F(\tau_2,\tau_2)\D\tau_2 = \int_0^{h}\ee^{(h-\tau_2)\mathcal{L}}\alpha_{n}[t](\tau_2)\left(-ad_{\mathcal{L}}^1(\alpha_{n}[t](\tau_2))+\alpha_n[t]'(\tau_2)\right)\ee^{\tau_2\mathcal{L}}u[t](0)\D\tau_2.$$
Polynomial $p$ satisfies the conditions
\begin{eqnarray*}
    p(0,0) &=& F(0,0),\\
    p(0,h) &=& F(0,h), \\
    p(h,h)&=& F(h,h), \\
    \int_0^h \partial_{\tau_1}^1p(\tau_2,\tau_2)\D\tau_2 &=&\int_0^h \partial_{\tau_1}^1F(\tau_2,\tau_2)\D\tau_2.
\end{eqnarray*}
The Filon quadrature reads
\begin{eqnarray*}   &&\int_0^h\int_0^{\tau_2}F(\tau_1,\tau_2)\left(\ee^{\ii\omega n(\tau_1-\tau_2)}+\ee^{\ii\omega n(-\tau_1+\tau_2)}\right)\D\tau_1\D\tau_2 \approx 
\int_0^h\int_0^{\tau_2}p(\tau_1,\tau_2)\left(\ee^{\ii\omega n(\tau_1-\tau_2)}+\ee^{\ii\omega n(-\tau_1+\tau_2)}\right)\D\tau_1\D\tau_2.
\end{eqnarray*}
\\

\noindent \emph{Trivariate integral}

\noindent The last integral to be approximated is 
\begin{eqnarray*}
  \int_0^h \int_0^{\tau_3}\int_0^{\tau_2} F(\tau_1,\tau_2,\tau_3)\ee^{\ii\omega(\tau_1n_1+\tau_2n_2+\tau_3n_3)}\D\tau_1\D\tau_2\D\tau_3,
\end{eqnarray*}
where $F(\tau_1,\tau_2,\tau_3) = \ee^{(h-\tau_3)\mathcal{L}}\alpha_{n_3}[t](\tau_3)\ee^{(\tau_3-\tau_2)\mathcal{L}}\alpha_{n_2}[t](\tau_2)\ee^{(\tau_2-\tau_1)}\alpha_{n_1}[t](\tau_1)\ee^{\tau_1\mathcal{L}}u[t](0)
$.
Function $F$ is interpolated in the nodes $(0,0,0)$, $(0,0,h)$, $(0,h,h)$, $(h,h,h)$ and 
$F(0,0,0)=p(0,0,0)$, $F(0,0,h)=p(0,0,h)$, $F(0,h,h)=p(0,h,h)$, $F(h,h,h)=p(h,h,h)$.
$$F(\tau_1,\tau_2,\tau_3)\approx p(\tau_1,\tau_2,\tau_3)= F(0,0,0)+a_{3,1}\tau_1+a_{3,2}\tau_2+a_{3,3}\tau_3,$$
where
\begin{eqnarray*}
    a_{3,3} &=& \frac{1}{h}(F(0,0,h)-F(0,0,0)),\\
    a_{3,2} &=& \frac{1}{h}(F(0,h,h)-F(0,0,h)),\\
    a_{3,1} &=& \frac{1}{h}(F(h,h,h)-F(0,h,h)),
\end{eqnarray*}
and
\begin{eqnarray*}
    F(0,0,0) &=& \ee^{h\mathcal{L}}\alpha_{n_3}[t](0)\alpha_{n_2}[t](0)\alpha_{n_1}[t](0)u[t](0),\\
     F(0,0,h) &=& \alpha_{n_3}[t](h)\ee^{h\mathcal{L}}\alpha_{n_2}[t](0)\alpha_{n_1}[t](0)u[t](0),\\
      F(0,h,h) &=& \alpha_{n_3}[t](h)\alpha_{n_2}[t](h)\ee^{h\mathcal{L}}\alpha_{n_1}[t](0)u[t](0),\\
      F(h,h,h) &=& \alpha_{n_3}[t](h)\alpha_{n_2}[t](h)\alpha_{n_1}[t](h)\ee^{h\mathcal{L}}u[t](0).
\end{eqnarray*}
Approximation of the trivariate integral is
\begin{eqnarray*}
  &&\int_0^h \int_0^{\tau_3}\int_0^{\tau_2} F(\tau_1,\tau_2,\tau_3)\ee^{\ii\omega(\tau_1n_1+\tau_2n_2+\tau_3n_3)}\D\tau_1\D\tau_2\D\tau_3 \approx\\
  &&\int_0^h \int_0^{\tau_3}\int_0^{\tau_2} (F(0,0,0)+a_{3,1}\tau_1+a_{3,2}\tau_2+a_{3,3}\tau_3)\ee^{\ii\omega(\tau_1n_1+\tau_2n_2+\tau_3n_3)}\D\tau_1\D\tau_2\D\tau_3.
\end{eqnarray*}
\\

\end{appendices}



\end{document}